\documentclass[10pt,a4paper]{amsart}

\usepackage[english]{babel}
\usepackage[utf8]{inputenc}
\usepackage{mathtools}
\usepackage{amsmath}
\usepackage{amsthm}
\usepackage{amssymb}
\usepackage{amsfonts}
\usepackage{array}
\usepackage{multirow}
\usepackage{amssymb}
\usepackage{ccaption}
\usepackage{multicol}
\usepackage{float}
\usepackage{enumitem}
\usepackage{xfrac}
\usepackage[linesnumbered,ruled,vlined]{algorithm2e}
\usepackage{xypic}
\usepackage{caption}
\usepackage{graphicx,color} 
\usepackage{graphbox}
\usepackage{stmaryrd}
\usepackage{listings}
\usepackage{pdfpages}
\usepackage{cite}
\usepackage[vario]{fancyref}

\usepackage{tikz}
\usepackage{tikz-cd}
\usepackage{pgfplots}
\usetikzlibrary{calc}
\usetikzlibrary{decorations.pathmorphing}

\newtheorem{thm}{Theorem}

\newtheorem{conj}{Conjecture}

\newtheorem{cor}{Corollary}[section]
\newtheorem*{cor*}{Corollary}
\newtheorem{prop}{Proposition}[section]
\newtheorem{lem}{Lemma}
\newtheorem{rem}{Remark}

\theoremstyle{definition}

\DeclareMathOperator{\Gram}{Gram}

\DeclareMathOperator{\intt}{int}
\DeclareMathOperator{\ext}{ext}

\definecolor{mygreen}{rgb}{0,0.6,0}
\definecolor{mygray}{rgb}{0.898,0.898,0.898}
\newcommand\opac{.1}
\definecolor{mylgray}{rgb}{0.93,0.93,0.93}
\definecolor{mydgray}{rgb}{0.8,0.8,0.8}
\definecolor{mymauve}{rgb}{0.58,0,0.82}

\newcommand{\LL}{\mathcal{L}}

\newcommand{\rone}{\mathbb{R}}

\newcommand{\rth}{\mathbb{R}^3}

\newcommand{\rd}{\mathbb{R}^d}
\newcommand{\wrd}{\widehat{\mathbb{R}^d}}

\newcommand{\ed}{\mathbb{R}^{d+1,1}}

\newcommand{\stw}{\mathbb{S}^2}

\newcommand{\sd}{\mathbb{S}^d}

\usepackage{soul,xcolor}

\newcommand{\NE}{(1,1)}
\newcommand{\NW}{(-1,1)}
\newcommand{\SW}{(-1,-1)}
\newcommand{\SE}{(1,-1)}

\begin{document}

\title{Ball packings for links} 
\author{Jorge L. Ram\'irez Alfons\'in}
\address{UMI2924 J.-C. Yoccoz, CNRS-IMPA, Brazil and IMAG, Univ.\ Montpellier, CNRS, Montpellier, France}
\email{jorge.ramirez-alfonsin@umontpellier.fr}
\author{Ivan Rasskin}
\address{IMAG, Univ.\ Montpellier, CNRS, Montpellier, France}
\email{ivan.rasskin@umontpellier.fr}

\subjclass[MSC 2010]{52C17, 57K10}
\keywords{Ball packings, Links, Lorentzian space}

\begin{abstract} 

The \textit{ball number} of a link $L$, denoted by $\mathsf{ball}(L)$, is the minimum number of solid balls (not necessarily of the same size) needed to realize a necklace representing $L$. In this paper, we show that $\mathsf{ball}(L)\leq 5 \mathsf{cr}(L)$  where $\mathsf{cr}(L)$ denotes the \textit{crossing number} of $L$. To this end, we use the connection of the Lorentz geometry with the ball packings. The well-known Koebe-Andreev-Thurston circle packing Theorem is also an important brick for the proof.
Our approach yields to an algorithm to construct explicitly the desired necklace representation of $L$ in $\rth$.
\end{abstract}

\maketitle

\section{Introduction}


A {\em link} with $n$ components consists of $n$ disjoint simple closed curves in $\mathbb{R}^3$. A {\em knot} $K$ is a link with one component (we refer the reader to \cite{Adam} for standard background on knot theory).  A {\em link diagram} of a link $L$ is a regular projection of $L$ into $\mathbb{R}^2$ in such a way that the projection of each component is smooth and at most two curves intersect at any point. At each crossing point of the link diagram the curve which goes over the other is specified, see Figure \ref{fig1}. The \textit{crossing number} of a $L$, denoted by $\mathsf{cr}(L)$, is the minimum number on crossings among all the diagrams of links which are ambient isotopic to $L$.

\begin{figure}[H]
\centering
\includegraphics[width=.4\linewidth]{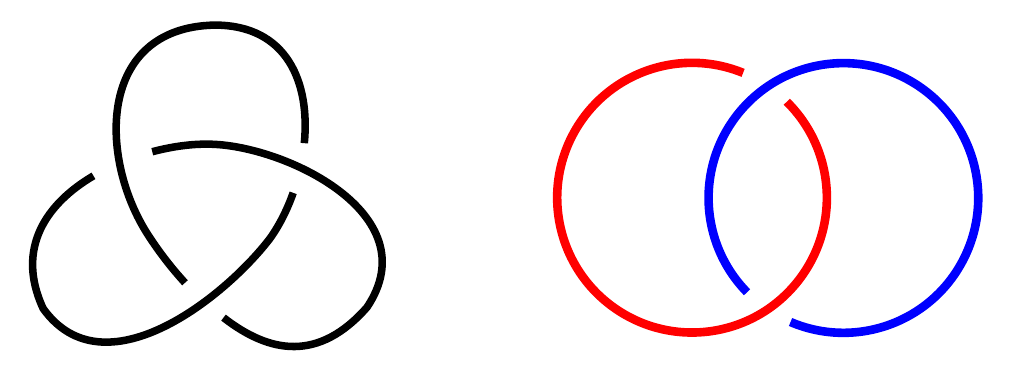}
\caption{(Left) A knot diagram of the \textit{trefoil} (denoted by $3_1$): the simplest non-trivial knot. (Right) A link diagram of the \textit{Hopf link} (denoted by $2_1^2$) : the simplest non-trivial link.}
\label{fig1}
\end{figure}

 A \textit{chain of balls}  is a sequence of non-overlapping solid balls in the space where all the consecutive balls are tangent.  The \textit{thread} of a chain of balls is the polygonal curve formed by joining the centers of consecutive tangent balls with straight segments. A chain of balls is\textit{ closed} if the last ball is tangent to the first ball. The thread of a closed chain can be thought of as a  polygonal knot in the space. A \textit{necklace representation} of a link $L$ is a collection of non-overlapping chains of balls such that theirs threads form a polygonal link ambient isotopic to $L$. 

\begin{figure}[H]
  \centering
    \includegraphics[width=0.37\textwidth]{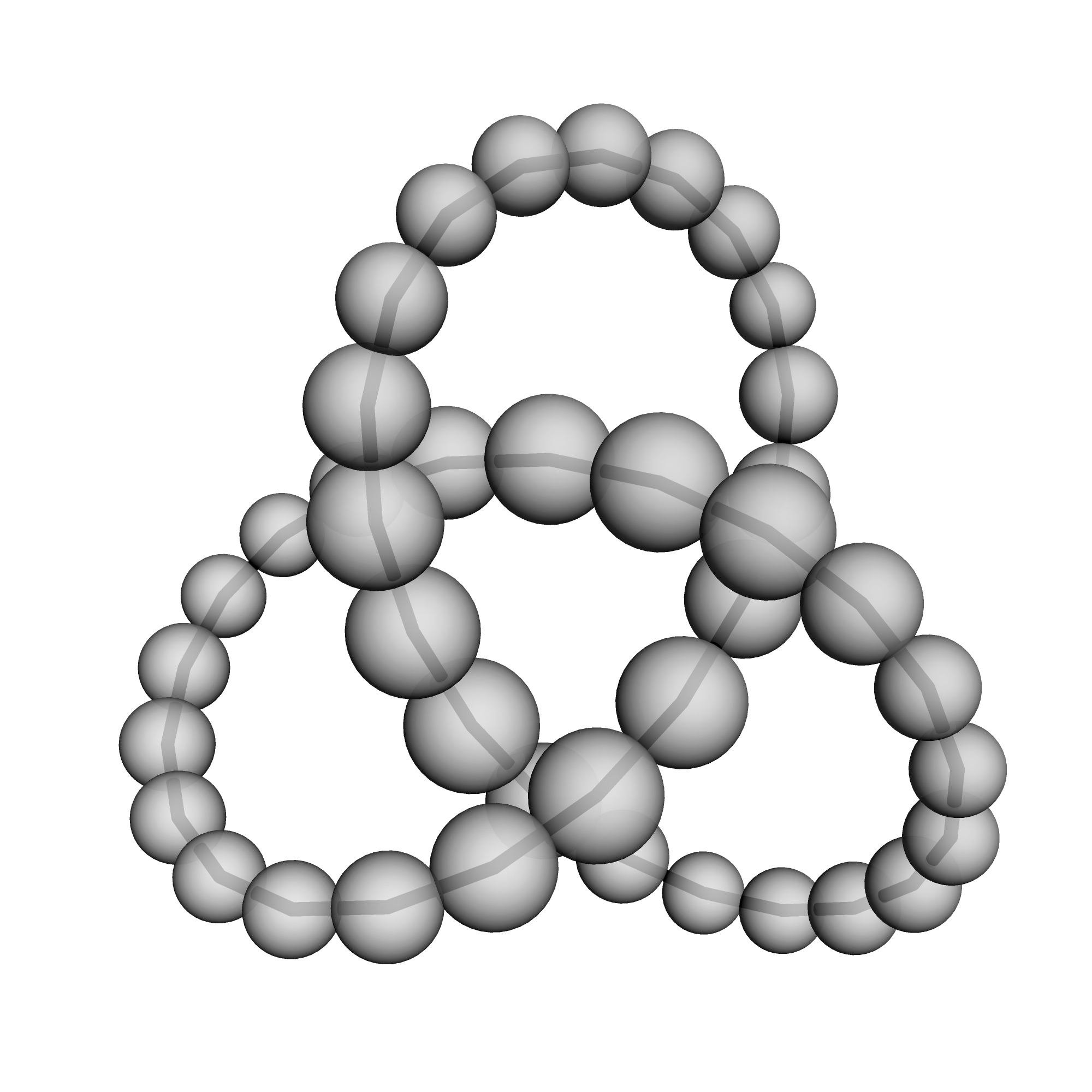}
    \caption{A necklace representation of the trefoil.}
  \label{necklace31}
\end{figure}

In \cite{mae2}, Maehara defined the \textit{ball number} of a link $L$, denoted by $\mathsf{ball}(L)$, as the minimum number of balls (not necessarily of the same size) needed to construct a necklace representation of $L$.
Little is known about the behavior of $\mathsf{ball}(L)$.  Maehara proved that $9\leq \mathsf{ball}(3_1)$ in \cite[Theorem 9]{mae2}. Some years earlier, Maehara and Oshiro showed that $\mathsf{ball}(3_1)\leq 12$ in \cite[Theorem 5]{mae3} and that $\mathsf{ball}(2_1^2)=8$ in \cite[Theorem 2]{mae1}. As far as we are aware, these are the only known results concerning ball numbers of links.\\

Necklace representations can be regarded as a particular case of polygonal representations of links with a strong geometric condition. Polygonal representations of links have been of great interest not only in mathematics but also in chemistry and physics. Indeed, polygonal representations have been applied to the study of the DNA and knotted molecules \cite{elrifai}.\\

In this paper we present the following upper bound to $\mathsf{ball}(L)$ in terms of the crossing number.

\begin{thm}\label{theorem} Let $L$ be a link. Then, 
$$\mathsf{ball}(L)\le 5\mathsf{cr}(L).$$
\end{thm}
\bigskip

Our approach allows to come up with an algorithm to construct explicitly the necklace representing the link $L$. We are able to compute the coordinates of the centers and the radius of each of the balls of the desired necklace.
\bigskip

Unfortunately, our technique does not allow us to push further the above upper bound. However, we believe that it can be improved.

\begin{conj}Let $L$ be a link. Then,
$$\mathsf{ball}(L)\le 4\mathsf{cr}(L).$$
Moreover, the equality holds if $L$ is alternating.
\end{conj}

Given the connection of the ball number with the Koebe-Andreev-Thurston circle packing Theorem (KAT Theorem) a linear bound seems inevitable.
\bigskip

A close related invariant to the ball number is the {\em pearl number} of a link $L$ in which the balls  have to be of the same size. The latter seems closely connected to another geometric invariant, the {\em rope length} of $L$. There is a known quasilinear upper bound on rope length in terms of the crossing number and a linear upper bound is conjectured. In \cite{DiaoErnst}, a sequence of knots is given that their rope length grows linearly in the crossing number. It is natural to ask whether the latter can be refined to pearl necklaces with unequal pearls. 

\bigskip

The paper is self-contained and it is organized as follows. In the next section, we briefly introduce some basic notions on the space of $d$-balls. We then explain the connection of the Lorentz geometry with the space of $d$-balls. We also discuss some definitions and properties of the inverse product and the action of the M\"obius group on the space of $d$-balls.
\bigskip

In Section \ref{sec:packings}, after recalling classical background of ball packing theory we introduce and study some geometric properties of both \textit{pyramidal disk systems} and \textit{crossing ball systems}. These are two building blocks for our construction.
\bigskip

In Section \ref{sec:proof}, we prove our main result. Let us give a brief outline of the proof. By combining the projection of a link with its associated {\em medial graph} we construct a simple planar graph which contains a subgraph isotopic to the projection of the given link. By the KAT Theorem we can obtain a disk packing whose tangency graph is the previous simple planar graph. We then construct a ball packing with same tangency graph as the disk packing. We finally obtain a necklace representation of the link by adding two balls to the ball packing for each crossing of the projection. We use the Lorentz geometry and the building blocks defined in Section 3 in order to verify that our construction works properly. 
\bigskip

In Section \ref{sec:alg}, based on the approach used in the proof of Theorem \ref{theorem}, we will present an algorithm that outputs the coordinates of the centers and the radius of the balls forming the necklace representation of the given link $L$. The examples presented at the end of this paper have been done throughout an implementation of this  algorithm.
\newpage
\section{The space of $d$-balls.}\label{sec:ballspace}
\subsection{From spherical caps to $d$-balls}
 Some notations and definitions of this section can be found in the PhD thesis of Chen \cite{chentesis} and the paper of Wilker \textit{Inversive geometry} \cite{wilker}. Let $d\geq 1$ be an integer. We denote $\rd$ the Euclidean space of dimension $d$ and $\langle\cdot,\cdot\rangle_2$, $\|\cdot\|$ the Euclidean inner product and the Euclidean norm respectively.
Let $\sd$ be the unit $d$-sphere of $\mathbb{R}^{d+1}$ endowed with the induced metric $\|\cdot\|_\mathbb{S}$ from $\mathbb{R}^{d+1}$. A \textit{$d$-spherical cap} $\beta$ of center $\gamma\in\sd$ and spherical radius $\rho\in(0,2\pi)$ is the subset 
\begin{eqnarray}\label{scap}
\beta=\{x\in\sd\mid \|x-\gamma\|_\mathbb{S}\leq \rho\}
\end{eqnarray}
which gives a partition of $\sd$ in three disjoint subsets: the \textit{interior} of $\beta$,  $\intt(\beta)$, points of $\sd$ satisfying (\ref{scap}) strictly, the \textit{exterior} of $\beta$, $\ext(\beta)$, points of $\sd$ not satisfying (\ref{scap}) and the boundary of $\beta$, $\partial \beta$, points of $\sd$ satisfying the equality of (\ref{scap}).
Let $\mathsf{Caps}(\sd)$ denote the family of $d$-spherical caps. It is well known that $\sd$ is homeomorphic to $\wrd$ under the stereographic projection where $\wrd:=\rd\cup\{\infty\}$ is the one-point compactification of $\rd$. A \textit{$d$-ball} of $\wrd$ is the image of a $d$-spherical cap under the stereographic projection. We denote $\mathsf{Balls}(\wrd)$ the space of $d$-balls, isomorphic to $\mathsf{Caps}(\sd)$ given by the above construction. A $d$-ball $b$ is called \textit{solid ball, hollow ball} and \textit{half-space} depending on whether the pole of the stereographic projection lies in the exterior, interior or boundary of the corresponding $d$-spherical cap $\beta_b$. More precisely, a \textit{$d$-ball} of $\wrd$ of curvature $\kappa\in\rone$ will be one of the following subsets:
\begin{itemize}
\item[-]\textit{Solid ball}: $\{x\in \wrd\mid\|x-c\|\leq 1/\kappa\}$ when $\kappa>0$.\\
It is also a standard  $d$-ball of $\rd$ with center $c\in\rd$ and radius $\frac{1}{k}$.
\item[-]\textit{Hollow ball}:$\{x\in \wrd\mid\|x-c\|\geq-1/\kappa\}$ when $\kappa<0$.\\
It can be regarded as the closure of the exterior of a solid ball with its boundary. 
\item[-]\textit{Half-space}: $\{x\in \wrd\mid\langle x,n\rangle_2\leq \delta\}$ when $\kappa=0$.\\
By convention, we choose the \textit{normal vector} $n$ which points towards the interior. The real number $\delta$ represents the \textit{signed distance} from the boundary to the origin (positive if the origin is contained in the interior and negative otherwise).
\end{itemize}
There is a natural embedding of $\mathsf{Balls}(\wrd)\hookrightarrow\mathsf{Balls}(\widehat{\mathbb{R}^{d+1}})$  where a $d$-ball $b$ of center $c$ and curvature $\kappa$ (resp. normal vector $n$ and signed distance $\delta$) is mapped to a $(d+1)$-ball $\widehat{b}$ of center $(c,0)$ and curvature $\kappa$ (resp. normal vector $(n,0)$ and signed distance $\delta$). We call this mapping the \textit{blowing up.}

\subsection{The intersection angle of two $d$-balls}
For $d>1$, let $b$ and $b'$ be two $d$-balls with intersecting boundaries. We define the \textit{intersection angle} of $b$ and $b'$, denoted by $\measuredangle(b,b')\in [0,\pi]$, as the angle formed by the vectors $\overrightarrow{pc}$ and $\overrightarrow{pc}'$ where $c$ and $c'$ are the centers of $b$ and $b'$ and $p\in\partial b\cap \partial b'$, see Figure \ref{fig:ang}. The intersection angle does not depend on the choice of the point in the intersection.
\begin{figure}[H]
 \centering

\begin{tikzpicture}[scale=1.2]
\fill[darkgray,opacity=.1] (-.92,-.39) circle (1cm);
\draw (-.92,-.39) circle (1cm);
\node (c) at (-.92,-.39) [circle,fill=black,inner sep=0pt,minimum size=0.1cm] {};
\fill[darkgray,opacity=.1] (.58,-0.39) circle (.7cm);
\draw (.58,-0.39) circle (.7cm);
\node (c') at (.58,-0.39) [circle,fill=black,inner sep=0pt,minimum size=0.1cm] {};
\node (p) at (0,0) [circle,fill=black,inner sep=0pt,minimum size=0.1cm] {};
\draw[->,>=stealth] (p) -- (c);
\draw[->,>=stealth] (p) -- (c');
\draw [<->,>=stealth,thick,domain=-34.05:-156.93] plot ({.3*cos(\x)}, {.3*sin(\x)});

\end{tikzpicture}    
    
    \caption{The intersection angle of two disks.}
    \label{fig:ang}
\end{figure}
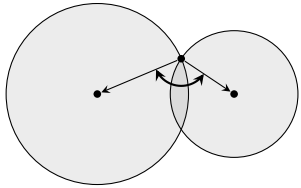
Two $d$-balls $b$ and $b'$ with intersecting boundaries are said to be:
\begin{itemize}
\item[-]\textit{Internally tangent} if $\measuredangle(b,b')=0$.
\item[-]\textit{Orthogonal} if $\measuredangle(b,b')=\frac{\pi}{2}$.
\item[-]\textit{Externally tangent} if $\measuredangle(b,b')=\pi$.

\end{itemize}

When the boundaries of $b$ and $b'$ do not intersect the intersection angle $\measuredangle(b,b')$ is not well-defined. In this case we say that $b$ and $b'$ are \textit{disjoint} if they have disjoint interiors and \textit{nested} if one is contained in the other.

\begin{rem}The blowing up operation preserves intersection angles.
\end{rem}

We notice that the definition of intersection angle does not work when $d=1$ since the boundary of a $1$-ball is not simply connected. In this case we can define the intersection angle of two $1$-balls as the intersection angle of the corresponding $2$-balls given by the blowing up.

\subsection{The hyperbolic model of $\mathsf{Balls}(\wrd)$}
Let $\mathbb{H}^{d+1}$ be the Poincaré ball model of the hyperbolic space of dimension $d+1$ embedded in  $\widehat{\mathbb{R}^{d+1}}$ as the standard unit $(d+1)$-ball. The boundary $\partial\mathbb{H}^{d+1}$ is exactly the unit sphere $\sd$. A  \textit{$d$-hyperbolic half-space} of $\mathbb{H}^{d+1}$ is the intersection 
$h:=\mathbb{H}^{d+1}\cap \widehat{b}_h$
where $ \widehat{b}_h$ is a $(d+1)$-ball orthogonal to $\mathbb{H}^{d+1}$. 
We denote $\mathsf{Halfs}(\mathbb{H}^{d+1})$ the space of hyperbolic half-spaces of $\mathbb{H}^{d+1}$. At the boundary of $\mathbb{H}^{d+1}$, the intersection $\partial \mathbb{H}^{d+1}\cap\widehat{b}_h=\sd\cap\widehat{b}_h$ is a $d$-spherical cap $\beta_h$ which corresponds to a $d$-ball $b_h$ by the stereographic projection. For any $d$-ball, the mapping $h\mapsto \beta_h\mapsto b_h$ can be reversed so we can define the following isomorphisms:
\begin{eqnarray}
\begin{tikzcd}
\mathsf{Balls}(\wrd)\arrow{r}{\simeq}& \mathsf{Caps}(\sd)\arrow{r}{\simeq}&\mathsf{Halfs}(\mathbb{H}^{d+1})
\end{tikzcd}
\label{eq:isos}
\end{eqnarray}
The notions of interior, exterior and boundary are easily extended for $d$-hyperbolic half-spaces. For $d>1$, two $d$-balls $b$ and $b'$ have intersecting boundaries if and only if the corresponding $d$-hyperbolic half-spaces  $h_b$ and $h_{b'}$ have intersecting boundaries. Moreover, the intersection angle of $b$ and $b'$ is equal to the \textit{dihedral angle} of $h_b$ and $h_{b'}$ measured at a  non-common region.

\subsection{The Lorentzian model of $\mathsf{Balls}(\wrd)$}
The Lorentzian space of dimension $d+2$, denoted by $\ed$, is a real vector space of dimension $d+2$ equipped with a bilinear symmetric form $\langle\cdot,\cdot\rangle$ of signature $(d+1,1)$. The \textit{Lorentzian product} of two vectors $u$ and $v$ of $\ed$ is the real number $\langle u,v\rangle$ and the \textit{Gramian} of a collection of vectors  $\mathcal{B}=\{v_1,\ldots,v_n\}$ of  $\ed$ is the matrix
\begin{eqnarray*}\label{eq:cartprod}
\Gram(\mathcal{B})&:=&
\begin{pmatrix}
\langle v_1,v_1\rangle  & \cdots & \langle v_1,v_n \rangle  \\
\vdots & \ddots & \vdots \\
\langle v_n,v_1 \rangle & \cdots & \langle v_n,v_n \rangle 
\end{pmatrix}
\end{eqnarray*}
If  $\mathcal{B}=\{v_1,\ldots,v_{d+2}\}$ is a basis of $\ed$ then $\Gram(\mathcal{B})$ is the matrix of the Lorentzian product in the basis $\mathcal{B}$. The Lorentzian product of two vectors $u,v\in\ed$ can be computed by 
\begin{eqnarray}\label{eq:carGram}
\langle u,v \rangle = C_{\mathcal{B}}(u)^T \Gram(\mathcal{B}) C_{\mathcal{B}}(v)
\end{eqnarray}
where $C_{\mathcal{B}}(u)$ is the column matrix of the coordinates of $u$ in the basis $\mathcal{B}$. Generalizing the definition of  Boyd in \cite{boyd},  we define the \textit{polyspherical coordinates} of a vector $u\in \ed$ with respect to $\mathcal{B}$ as the column matrix $P_{\mathcal{B}}(u)=
\begin{pmatrix}
\langle v_1,u\rangle   &
\cdots &
\langle v_{d+2},u\rangle
\end{pmatrix}^T$ which is related to the Cartesian coordinates by
\begin{eqnarray}\label{eq:carpol}
C_{\mathcal{B}}(u)= \Gram(\mathcal{B})^{-1} P_{\mathcal{B}}(u)
\end{eqnarray}
Combining equations (\ref{eq:carGram}) and (\ref{eq:carpol}) we can compute the Lorentzian product in polyspherical coordinates by
\begin{eqnarray}\label{eq:polyprod}
\langle u,v\rangle = P_{\mathcal{B}}(u)^T \Gram(\mathcal{B})^{-1} P_{\mathcal{B}}(v)
\end{eqnarray}
In practice we will use equation (\ref{eq:polyprod}) to compute the Lorentzian product in different basis.

\noindent
From now on, we fix an orthonormal basis $\mathcal{B}_0=\{e_1,\ldots,e_{d+2}\}$ with Gramian $\mathsf{diag}(1,\ldots,1,-1)$.
A vector $v\in\ed$ is called:
\begin{itemize}
\item[-]\textit{Space-like} (resp. \textit{time-like}) if $\langle v,v\rangle >0$ (resp. $<0$). 
\item[-]\textit{Future-directed} (resp. \textit{past-directed}) if $\langle e_{d+2},v\rangle>0$ (resp. $<0$).
\item[-]\textit{Normalized} if $|\langle v,v\rangle|=1$.
\end{itemize}
The space of all the normalized space-like (resp. time-like) vectors of $\ed$ is usually called  \textit{de Sitter space} (resp. \textit{anti de Sitter space}). We denote it by $\mathsf{S}(\ed)$ (resp. $\mathsf{T}(\ed)$). The \textit{anti de Sitter space} can be regarded as the generalization of a two-sheets hyperboloid with two connected components  $\mathsf{T}^\uparrow(\ed)$ and $\mathsf{T}^\downarrow(\ed)$ formed by the future-directed and the past-directed vectors of $\mathsf{T}(\ed)$ respectively. The \textit{hyperboloid model} of the $(d+1)$-hyperbolic space is obtained by taking $\mathsf{T}^\uparrow(\ed)$ with the metric induced by the restriction of the Lorentzian product  of $\ed$. The isomorphism which maps the hyperboloid model to the Poincar\'e ball model can be regarded as the projection $\Pi:\mathsf{T}^\uparrow(\ed)\rightarrow \{e_{d+2}=0\}$ from $-e_{d+2}$, see Figure \ref{eq:isos}. \\

A \textit{time-like half-space} is the subset $t_v=\{u\in\ed\mid\langle u,v\rangle\geq0\}$
where $v\in\mathsf{S}(\ed)$. The space of time-like half-spaces is in bijection with $\mathsf{S}(\ed)$.  The image $\Pi(t_v\cap\mathsf{T}^\uparrow(\ed))$ is a hyperbolic half-space of $\mathbb{H}^{d+1}$ and every hyperbolic half-space can be obtained in this way. We can then extend the isomorphisms of (\ref{eq:isos})  by
\begin{eqnarray}
\begin{tikzcd}
\mathsf{Balls}(\wrd)\arrow{r}{\simeq}& \mathsf{Caps}(\sd)\arrow{r}{\simeq}&\mathsf{Halfs}(\mathbb{H}^{d+1})\arrow{r}{\simeq}&\mathsf{S}(\ed)
\end{tikzcd}
\label{eq:isos2}
\end{eqnarray}

The \textit{Lorentzian vector} of a $d$-ball $b$, denoted $v_b$, is the normalized space-like vector obtained by the previous isomorphisms. 

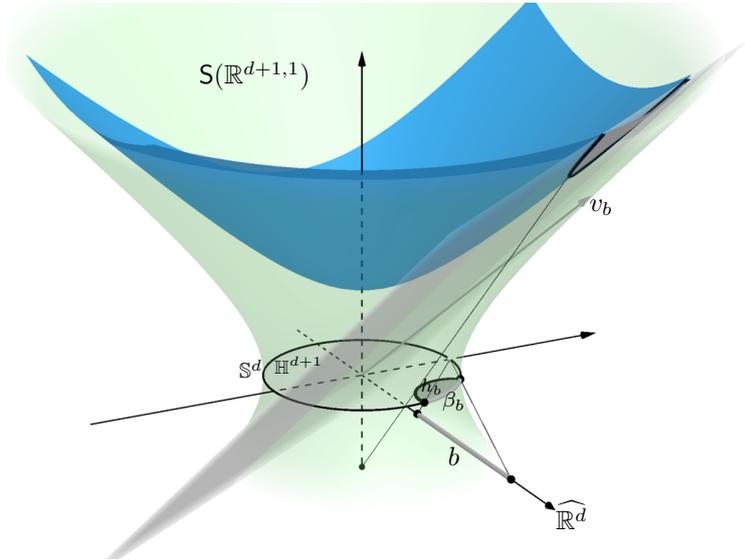
\begin{figure}[H]
 \centering
\begin{tikzpicture}[scale=.97]
\clip (-6,-3.6) rectangle (6,4);
\node {\includegraphics[scale=1.3]{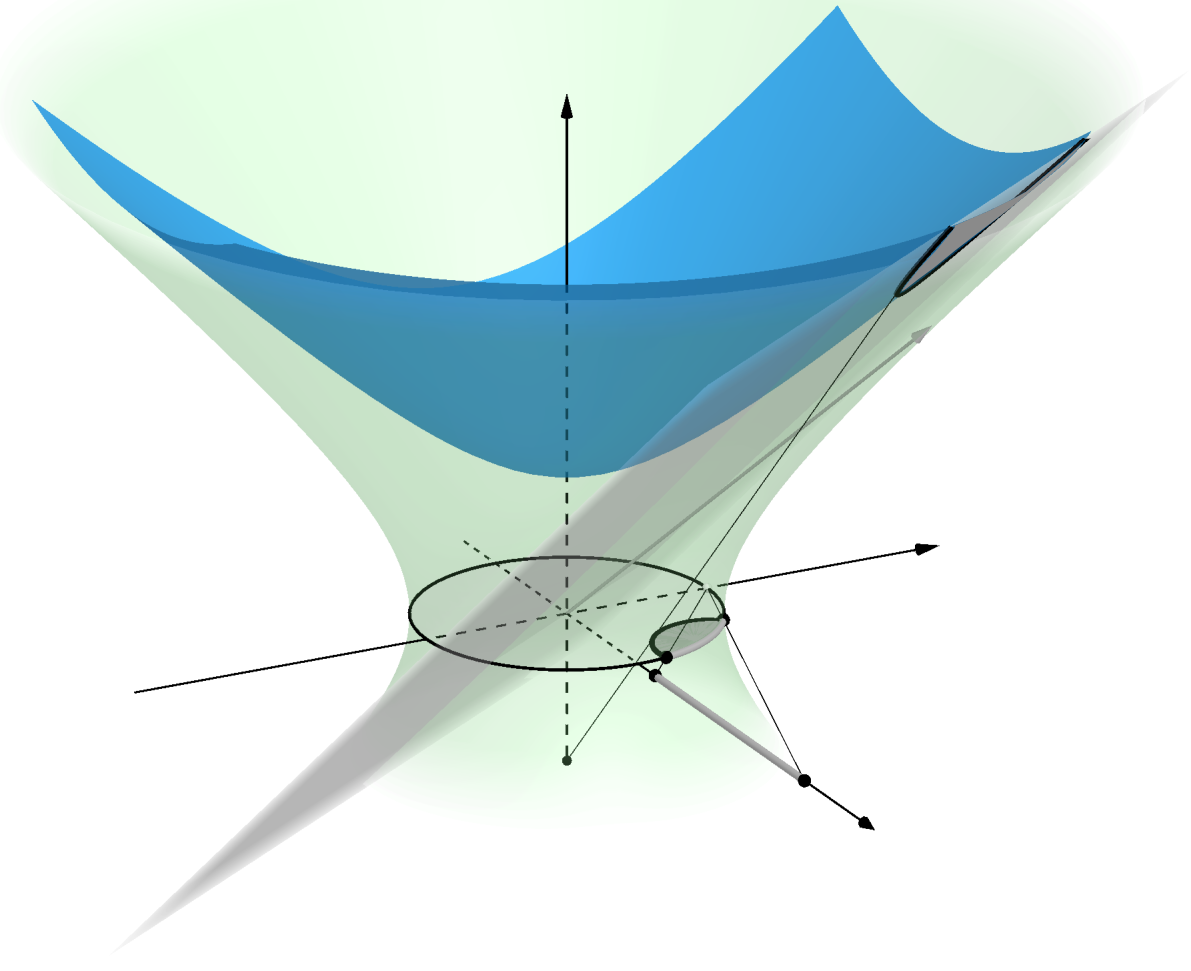}};
\node at (2.6,-3) {$\widehat{\mathbb{R}^d}$};
\node at (-1.7,3) {$\mathsf{S}(\mathbb{R}^{d+1,1})$};
\node at (-1.75,-1) {\footnotesize $\mathbb{S}^d$};
\node at (-1.1,-.95) {\scriptsize $\mathbb{H}^{d+1}$};
\node at (1,-2.2) {$b$};
\node at (3,1.2) {$v_b$};
\node at (1,-1.45) {\footnotesize $\beta_b$};
\node at (.7,-1.25) {\scriptsize  $h_b$};
\end{tikzpicture}

    \caption{Geometric interpretation of the isomorphisms of Equation (\ref{eq:isos2}).}
    \label{fig:isos}
\end{figure}
The \textit{inversive product} of two $d$-balls $b$ and $b'$,  denoted by  $\langle b,b'\rangle:=\langle v_b, v_b'\rangle$, is the Lorentzian product of their corresponding Lorentzian vectors. Equivalently, we define the Gramian of a collection of $d$-balls as the Gramian of the collection of the Lorentzian vectors of the $d$-balls. We denote
 $-b$ the $d$-ball corresponding to the Lorentzian vector $-v_b$ which is the $d$-ball with same boundary as $b$ and $\intt(-b)=\ext(b)$. We notice that  $\langle -b,b'\rangle=-\langle b,b'\rangle$. The inversive product is a fundamental tool to encode configurations of $d$-balls \cite{wilker}. Indeed,
\begin{eqnarray*}
\langle b,b'\rangle&=&
\left\lbrace\begin{array}{lll}
\cosh d_\mathbb{H}(h_{b},h_{b'})&\text{if $b$ and $b'$ are nested} \\
\cos \measuredangle (b,b')&\text{if $\partial b$ and $\partial b'$ intersect}\\
-\cosh d_\mathbb{H} (h_{b},h_{b'})&\text{if $b$ and $b'$ are disjoint} \\
\end{array}\right.
\end{eqnarray*}
where $h_b$ and $h_{b'}$  are the corresponding hyperbolic half-spaces and $d_\mathbb{H}(h_b,h_{b'})$ is the hyperbolic distance between $\partial h_b$ and $\partial h_{b'}$. In particular, we have
\begin{eqnarray}\label{eq:lprod}
\langle b,b'\rangle&=&
\left\lbrace\begin{array}{rll}
>1 \quad&\text{if  $b$ and $b'$ are nested} \\
1 \quad&\text{if $b$ and $b'$ are internally tangent}  \\
0 \quad&\text{if  $b$ and $b'$ are orthogonal}  \\
-1 \quad&\text{if  $b$ and $b'$ are externally tangent}  \\
<-1 \quad&\text{if  $b$ and $b'$ are disjoint} 
\end{array}\right.
\end{eqnarray}

In \cite{wilker}, Wilker defined the \textit{inversive coordinates} of a $d$-ball $b$ as the column-matrix of the Cartesian coordinates of $v_b$ with respect to $\mathcal{B}_0$. We overuse the notation $v_b$ to refer both the Lorentzian vector and the inversive coordinates of the $d$-ball $b$. The latter can be given in terms of the curvature $\kappa$ and center $c$ (normal vector $n$ and signed distance $\delta$ for half-spaces) by
\begin{eqnarray}
v_b&=&
\left\lbrace\begin{array}{lll}
\dfrac{\kappa}{2}(2c,\|c\|^2-\frac{1}{\kappa^2}-1,\|c\|^2-\frac{1}{\kappa^2}+1)^T&\text{ if }\kappa\not=0,  \\
\quad\\

(n,\delta,\delta)^T&\text{ if }\kappa=0. \\
\end{array}\right.
\end{eqnarray}
With the inversive coordinates one can compute the inversive product by
\begin{eqnarray}\label{eq:invproduct}
\langle b,b'\rangle=v_b^T Q\, v_{b'}
\end{eqnarray} 
where $Q=\mathsf{diag}(1,\cdots,1,-1)$ is the Gramian of $\mathcal{B}_0$.
\subsection{The Möbius group}
The Möbius Group $\mathsf{M\ddot{o}b}(\wrd)$ can be defined as the group of the continuous automorphisms of $\wrd$ mapping $d$-balls to $d$-balls \cite{thurston}. An element of the Möbius Group is called a \textit{Möbius transformation.}
For every solid or hollow (resp. half-space) $d$-ball $b$ we define the \textit{inversion on} $b$, denoted by $\sigma_b$, as the sphere inversion (resp. Euclidean reflection) on $\partial b$. Alternatively, $\sigma_b$ can be defined as the only Möbius transformation which maps $b$ to $-b$ and fixes a $d$-ball $b'$ if and only if $b'$ is orthogonal to $b$. It is well-known that $\mathsf{M\ddot{o}b}(\wrd)$ is generated by the inversions on $d$-balls \cite{thurston}. The product $\sigma_b\circ\sigma_{b'}$ of the inversions on two $d$-balls centered at the origin with non-zero curvatures $\kappa$ and $\kappa'$  gives a scaling of $\mathbb{R}^d$ of scaling factor $(\kappa'/\kappa)^2$. Thus, the group of Euclidean isometries and scalings of $\mathbb{R}^d$ is a subgroup of $\mathsf{M\ddot{o}b}(\wrd)$.\\

The Möbius Group defines a group action (on the left) on the space of $d$-balls. The following group isomorphisms and equivariant group actions can be obtained by using the isomorphisms given in (\ref{eq:isos2})
\begin{eqnarray}\label{isogroups}
\begin{tikzcd}[row sep=small]
\mathsf{Balls}(\wrd)\arrow[loop below]\arrow{r}{\simeq}& \mathsf{Caps}(\sd)\arrow[loop below]\arrow{r}{\simeq}&\mathsf{Halfs}(\mathbb{H}^{d+1})\arrow[loop below]\arrow{r}{\simeq}&\mathsf{S}(\ed)\arrow[loop below]\\
\mathsf{M\ddot{o}b}(\wrd)\arrow{r}{\simeq}& \mathsf{M\ddot{o}b}(\mathbb{S}^d)\arrow{r}{\simeq}&\mathsf{Isom}(\mathbb{H}^{d+1})\arrow{r}{\simeq}&\mathsf{O}^\uparrow(\ed)
\end{tikzcd}
\end{eqnarray}
where $\mathsf{M\ddot{o}b}(\sd)$ is the Möbius group defined on $\mathbb{S}^d$ acting on the family of $d$-spherical caps,
$\mathsf{Isom}(\mathbb{H}^{d+1})$ is the group of hyperbolic isometries acting on the space of hyperbolic half-spaces and $\mathsf{O}^\uparrow(\ed)$ is the \textit{Orthochronous Lorentz Group} which is the group of linear maps of $\ed$ preserving the Lorentz product and the time orientation. The latter acts on the space of normalized space-like vectors of $\ed$. Moreover, for any $b\in\mathsf{Balls}(\wrd)$, the isomorphism $\mathsf{M\ddot{o}b}(\wrd)\rightarrow\mathsf{O}^\uparrow(\ed)$ maps the inversion $\sigma_b$ to the Lorentzian reflection on the boundary of the time-like half-space $t_{v_b}$, which corresponds to the linear map
\begin{eqnarray}\label{eq:lorenref}
\sigma_{v_b}:u\mapsto u-2\langle u,v_b\rangle v_b.
\end{eqnarray}
Since the Orthochronous Lorentz Group preserves the Lorentz product, the Möbius Group preserves the inversive product of $d$-balls.

\section{$d$-ball packings}\label{sec:packings}
A collection of $d$-balls $\mathcal{B}=\{b_1,\ldots,b_n\}$ in $\wrd$ is called a \textit{$d$-ball packing} if every pair of $d$-balls $b_i,b_j\in\mathcal{B}$ are either
externally tangent or disjoint. The \textit{tangency graph} of a $d$-ball packing $\mathcal{B}$ is the simple graph $G=(V,E)$ where 
$V=\{1,\ldots,n\}$ and $E=\{ij\mid b_i\text{ and }b_j\text{ are externally tangent}\}$. A simple graph $G$ is said to be \textit{$d$-ball packable} if there is a $d$-ball packing $\mathcal{B}_G$ with tangency graph $G$. In this case $G$ can be embedded in $\wrd$ by taking the centers of the $d$-balls of $\mathcal{B}_G$ and the straight segments between the centers of any tangent pair. This embedding is usually called the \textit{carrier} of the $d$-ball-packing, see Figure \ref{fig:skeleton}. The Möbius Group preserves  tangency graphs and maps carriers to carriers \cite{stephenson}.\\

\begin{figure}[H]
 \centering
    \begin{tikzpicture}[scale=.45] 

\draw (-0.000567384,1.32022) circle (0.47988cm);
\draw (-0.514838,0.709638) circle (0.318417cm);
\draw (0.207697,0.28587) circle (0.206672cm);
\draw (-0.205456,0.285608) circle (0.206481cm);
\draw (0.516768,0.71127) circle (0.319151cm);
\draw (-1.2499,0.405167) circle (0.477205cm);
\draw (-4.12964,1.3203) circle (2.54445cm);
\draw (-2.52836,-3.48) circle (2.51588cm);
\draw (-0.770035,-1.05986) circle (0.475567cm);
\draw (-0.830163,-0.269342) circle (0.317237cm);
\draw (-0.0128939,4.39213) circle (2.59206cm);
\draw (4.18115,1.34498) circle (2.59206cm);
\draw (2.53181,-3.51952) circle (2.54445cm);
\draw (1.25578,0.407429) circle (0.47988cm);
\draw (0.834,-0.270349) circle (0.318417cm);
\draw (0.771576,-1.06352) circle (0.477205cm);
\draw (0.000375084,-0.872763) circle (0.317237cm);
\draw (-0.332571,-0.10698) circle (0.206174cm);
\draw (0.00102559,-0.349352) circle (0.206174cm);
\draw (0.335119,-0.107142) circle (0.206481cm);

\fill[darkgray,opacity=\opac] (-0.000567384,1.32022) circle (0.47988cm);
\fill[darkgray,opacity=\opac] (-0.514838,0.709638) circle (0.318417cm);
\fill[darkgray,opacity=\opac] (0.207697,0.28587) circle (0.206672cm);
\fill[darkgray,opacity=\opac] (-0.205456,0.285608) circle (0.206481cm);
\fill[darkgray,opacity=\opac] (0.516768,0.71127) circle (0.319151cm);
\fill[darkgray,opacity=\opac] (-1.2499,0.405167) circle (0.477205cm);
\fill[darkgray,opacity=\opac] (-4.12964,1.3203) circle (2.54445cm);
\fill[darkgray,opacity=\opac] (-2.52836,-3.48) circle (2.51588cm);
\fill[darkgray,opacity=\opac] (-0.770035,-1.05986) circle (0.475567cm);
\fill[darkgray,opacity=\opac] (-0.830163,-0.269342) circle (0.317237cm);
\fill[darkgray,opacity=\opac] (-0.0128939,4.39213) circle (2.59206cm);
\fill[darkgray,opacity=\opac] (4.18115,1.34498) circle (2.59206cm);
\fill[darkgray,opacity=\opac] (2.53181,-3.51952) circle (2.54445cm);
\fill[darkgray,opacity=\opac] (1.25578,0.407429) circle (0.47988cm);
\fill[darkgray,opacity=\opac] (0.834,-0.270349) circle (0.318417cm);
\fill[darkgray,opacity=\opac] (0.771576,-1.06352) circle (0.477205cm);
\fill[darkgray,opacity=\opac] (0.000375084,-0.872763) circle (0.317237cm);
\fill[darkgray,opacity=\opac] (-0.332571,-0.10698) circle (0.206174cm);
\fill[darkgray,opacity=\opac] (0.00102559,-0.349352) circle (0.206174cm);
\fill[darkgray,opacity=\opac] (0.335119,-0.107142) circle (0.206481cm);

\node at (-0.000567384,1.32022) [circle,fill=black,inner sep=0pt,minimum size=0.1cm] {};
\node at (-0.514838,0.709638) [circle,fill=black,inner sep=0pt,minimum size=0.1cm] {};
\node at (0.207697,0.28587) [circle,fill=black,inner sep=0pt,minimum size=0.1cm] {};
\node at (-0.205456,0.285608) [circle,fill=black,inner sep=0pt,minimum size=0.1cm] {};
\node at (0.516768,0.71127) [circle,fill=black,inner sep=0pt,minimum size=0.1cm] {};
\node at (-1.2499,0.405167) [circle,fill=black,inner sep=0pt,minimum size=0.1cm] {};
\node at (-4.12964,1.3203) [circle,fill=black,inner sep=0pt,minimum size=0.1cm] {};
\node at (-2.52836,-3.48) [circle,fill=black,inner sep=0pt,minimum size=0.1cm] {};
\node at (-0.770035,-1.05986) [circle,fill=black,inner sep=0pt,minimum size=0.1cm] {};
\node at (-0.830163,-0.269342) [circle,fill=black,inner sep=0pt,minimum size=0.1cm] {};
\node at (-0.0128939,4.39213) [circle,fill=black,inner sep=0pt,minimum size=0.1cm] {};
\node at (4.18115,1.34498) [circle,fill=black,inner sep=0pt,minimum size=0.1cm] {};
\node at (2.53181,-3.51952) [circle,fill=black,inner sep=0pt,minimum size=0.1cm] {};
\node at (1.25578,0.407429) [circle,fill=black,inner sep=0pt,minimum size=0.1cm] {};
\node at (0.834,-0.270349) [circle,fill=black,inner sep=0pt,minimum size=0.1cm] {};
\node at (0.771576,-1.06352) [circle,fill=black,inner sep=0pt,minimum size=0.1cm] {};
\node at (0.000375084,-0.872763) [circle,fill=black,inner sep=0pt,minimum size=0.1cm] {};
\node at (-0.332571,-0.10698) [circle,fill=black,inner sep=0pt,minimum size=0.1cm] {};
\node at (0.00102559,-0.349352) [circle,fill=black,inner sep=0pt,minimum size=0.1cm] {};
\node at (0.335119,-0.107142) [circle,fill=black,inner sep=0pt,minimum size=0.1cm] {};

\draw (-0.000567384,1.32022) -- (-0.514838,0.709638);
\draw (-0.000567384,1.32022) -- (0.516768,0.71127);
\draw (-0.000567384,1.32022) -- (-0.0128939,4.39213);
\draw (-0.514838,0.709638) -- (-0.205456,0.285608);
\draw (-0.514838,0.709638) -- (-1.2499,0.405167);
\draw (0.207697,0.28587) -- (-0.205456,0.285608);
\draw (0.207697,0.28587) -- (0.516768,0.71127);
\draw (0.207697,0.28587) -- (0.335119,-0.107142);
\draw (-0.205456,0.285608) -- (-0.332571,-0.10698);
\draw (0.516768,0.71127) -- (1.25578,0.407429);
\draw (-1.2499,0.405167) -- (-4.12964,1.3203);
\draw (-1.2499,0.405167) -- (-0.830163,-0.269342);
\draw (-4.12964,1.3203) -- (-2.52836,-3.48);
\draw (-4.12964,1.3203) -- (-0.0128939,4.39213);
\draw (-2.52836,-3.48) -- (-0.770035,-1.05986);
\draw (-2.52836,-3.48) -- (2.53181,-3.51952);
\draw (-0.770035,-1.05986) -- (-0.830163,-0.269342);
\draw (-0.770035,-1.05986) -- (0.000375084,-0.872763);
\draw (-0.830163,-0.269342) -- (-0.332571,-0.10698);
\draw (-0.0128939,4.39213) -- (4.18115,1.34498);
\draw (4.18115,1.34498) -- (2.53181,-3.51952);
\draw (4.18115,1.34498) -- (1.25578,0.407429);
\draw (2.53181,-3.51952) -- (0.771576,-1.06352);
\draw (1.25578,0.407429) -- (0.834,-0.270349);
\draw (0.834,-0.270349) -- (0.771576,-1.06352);
\draw (0.834,-0.270349) -- (0.335119,-0.107142);
\draw (0.771576,-1.06352) -- (0.000375084,-0.872763);
\draw (0.000375084,-0.872763) -- (0.00102559,-0.349352);
\draw (-0.332571,-0.10698) -- (0.00102559,-0.349352);
\draw (0.00102559,-0.349352) -- (0.335119,-0.107142);
\end{tikzpicture}
    \caption{A 2-ball packing with its carrier.}
    \label{fig:skeleton}

\end{figure}
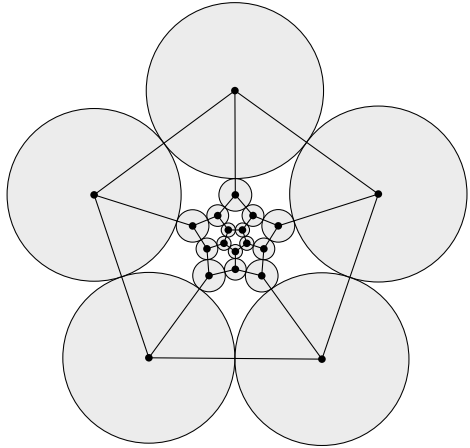

A d-ball packing $\mathcal{B}$ is said to be \textit{standard} if it contains the half-spaces $b_i=\{x_d\geq 1\}$ and $b_j=\{x_d\leq -1\}$. We denote this property by $[\mathcal{B}]^i_j$. We notice that the tangency point $b_i\cap b_j$ is at the infinity and the rest of the $d$-balls of $[\mathcal{B}]^i_j$ must lie inside the region $\{-1\leq x_d\leq 1\}$. For any d-ball packing $\mathcal{B}$ whose tangency graph contains at least one edge $ij$, a Möbius transformation $\phi:\mathcal{B}\mapsto[\mathcal{B}]^i_j$ will be called a \textit{standard transformation}. Standard transformations exist for every edge $ij$ of $G$ and they can be obtained as the product of an inversion in a $d$-ball centered at the tangency point $b_i\cap b_j$, a Euclidean isometry and a scaling of $\mathbb{R}^d$, see Figure \ref{fig;standard}.\\
\begin{figure}[H]
 \centering
\begin{tikzpicture}

\begin{scope}[scale=.12,xshift=-32cm,yshift=25cm]
\draw (-1.28409,0.741838) circle (0.793861cm);

\draw (-9.53931,-5.50753) circle (9.56004cm) node {$b_j$};
\draw (0.020156,10.7561) circle (9.305cm) node {$b_i$};
\node at (-4.69496, 2.73424) [circle,fill=black,inner sep=0pt,minimum size=.07cm]  {};
\draw[dashed] (-4.69496, 2.73424) circle (2.73cm);
\draw (0.000406956,-1.48297) circle (0.793861cm);
\draw (1.28018,0.739115) circle (0.790948cm);
\draw (-0.320324,0.18414) circle (0.319632cm);
\draw (-0.000259375,0.945177) circle (0.50597cm);
\draw (-0.820861,-0.473924) circle (0.50716cm);
\draw (9.32515,-5.36061) circle (9.305cm);
\draw (0.818417,-0.472813) circle (0.50597cm);
\draw (0.318647,0.183971) circle (0.319339cm);
\draw (-0.00069237,-0.369479) circle (0.319632cm);
\fill[darkgray,opacity=\opac] (-1.28409,0.741838) circle (0.793861cm);
\fill[darkgray,opacity=\opac] (-9.53931,-5.50753) circle (9.56004cm);
\fill[darkgray,opacity=\opac] (0.020156,10.7561) circle (9.305cm);
\fill[darkgray,opacity=\opac] (0.000406956,-1.48297) circle (0.793861cm);
\fill[darkgray,opacity=\opac] (1.28018,0.739115) circle (0.790948cm);
\fill[darkgray,opacity=\opac] (-0.320324,0.18414) circle (0.319632cm);
\fill[darkgray,opacity=\opac] (-0.000259375,0.945177) circle (0.50597cm);
\fill[darkgray,opacity=\opac] (-0.820861,-0.473924) circle (0.50716cm);
\fill[darkgray,opacity=\opac] (9.32515,-5.36061) circle (9.305cm);
\fill[darkgray,opacity=\opac] (0.818417,-0.472813) circle (0.50597cm);
\fill[darkgray,opacity=\opac] (0.318647,0.183971) circle (0.319339cm);
\fill[darkgray,opacity=\opac] (-0.00069237,-0.369479) circle (0.319632cm);
\end{scope}

\draw[->] (-1.5,3)--(1,3) node[midway,above,align=center]  {Standard\\transformation};

\begin{scope}[scale=1,xshift=4cm,yshift=3cm]
\draw (-2.8,1) -- (2.8,1) node[pos=.05,above] {$b_i$};
\fill[darkgray,opacity=\opac] (-2.8,1) rectangle (2.8,1.5);
\draw (-2.8,-1) -- (2.8,-1) node[pos=.05,below] {$b_i$};
\fill[darkgray,opacity=\opac] (-2.8,-1) rectangle (2.8,-1.5);
\draw (0.382,-0.618) circle (0.382cm);
\draw (1.62,0) circle (1.00cm);
\draw (0.382,0.618) circle (0.382cm);
\draw (-1.62,0) circle (1.00cm);
\draw (0,-0.191) circle (0.191cm);
\draw (-0.382,-0.618) circle (0.382cm);
\draw (0.382,0) circle (0.236cm);
\draw (-0.382,0.618) circle (0.382cm);
\draw (-0.382,0) circle (0.236cm);
\draw (0,0.191) circle (0.191cm);
\fill[darkgray,opacity=\opac] (0.382,-0.618) circle (0.382cm);
\fill[darkgray,opacity=\opac] (1.62,0) circle (1.00cm);
\fill[darkgray,opacity=\opac] (0.382,0.618) circle (0.382cm);
\fill[darkgray,opacity=\opac] (-1.62,0) circle (1.00cm);
\fill[darkgray,opacity=\opac] (0,-0.191) circle (0.191cm);
\fill[darkgray,opacity=\opac] (-0.382,-0.618) circle (0.382cm);
\fill[darkgray,opacity=\opac] (0.382,0) circle (0.236cm);
\fill[darkgray,opacity=\opac] (-0.382,0.618) circle (0.382cm);
\fill[darkgray,opacity=\opac] (-0.382,0) circle (0.236cm);
\fill[darkgray,opacity=\opac] (0,0.191) circle (0.191cm);
\end{scope}

\draw[->] (-3.84,1)--(-3.84,.25) node[midway,right]  {Inversion};

\begin{scope}[scale=.12,xshift=-32cm,yshift=-12cm]
\clip (-11,-6) rectangle (6,11);
\draw(-13,8.08373) -- (13,-7.19859);
\fill[darkgray,opacity=\opac](-13,8.08373)--(-13,13)--(13,13)--(13,-7.19859);
\draw (-13,7.16032) -- (13,-8.12199);
\fill[darkgray,opacity=\opac](-13,7.16032)--(-13,-10)--(13,-10)--(13,-8.12199);
\draw[dashed] (-4.69496, 2.73424) circle (2.73cm);
\node at (-1,8) {$b_i$};
\node at (-8,-3) {$b_j$};
\draw (-2.98477,1.73526) circle (0.398036cm);
\draw (-3.79573,1.92658) circle (0.152036cm);
\draw (-3.54642,2.35073) circle (0.152036cm);
\draw (-3.40893,1.98457) circle (0.0939636cm);
\draw (-3.28427,2.19665) circle (0.152036cm);
\draw (-3.53358,1.7725) circle (0.152036cm);
\draw (-4.09523,2.38797) circle (0.398036cm);
\draw (-3.67107,2.13866) circle (0.0939636cm);
\draw (-3.50148,2.12715) circle (0.0760182cm);
\draw (-3.57852,1.99608) circle (0.0760182cm);
\fill[darkgray,opacity=\opac] (-2.98477,1.73526) circle (0.398036cm);
\fill[darkgray,opacity=\opac] (-3.79573,1.92658) circle (0.152036cm);
\fill[darkgray,opacity=\opac] (-3.54642,2.35073) circle (0.152036cm);
\fill[darkgray,opacity=\opac] (-3.40893,1.98457) circle (0.0939636cm);
\fill[darkgray,opacity=\opac] (-3.28427,2.19665) circle (0.152036cm);
\fill[darkgray,opacity=\opac] (-3.53358,1.7725) circle (0.152036cm);
\fill[darkgray,opacity=\opac] (-4.09523,2.38797) circle (0.398036cm);
\fill[darkgray,opacity=\opac] (-3.67107,2.13866) circle (0.0939636cm);
\fill[darkgray,opacity=\opac] (-3.50148,2.12715) circle (0.0760182cm);
\fill[darkgray,opacity=\opac] (-3.57852,1.99608) circle (0.0760182cm);
\end{scope}

\draw[->] (-2.2,-1)--(2,-1) node[midway,above]  {Euclidean isometry};

\begin{scope}[scale=.12,xshift=37cm,yshift=-8cm,rotate=30]
\clip[rotate=-30] (-13,-7) rectangle (5,7);
\draw(-13,8.08373) -- (13,-7.19859) node[pos=.15,above=.1cm] {$b_i$};
\fill[darkgray,opacity=\opac](-13,8.08373)--(-13,13)--(13,13)--(13,-7.19859);
\draw (-13,7.16032) -- (13,-8.12199)node[pos=.135,below=.1cm] {$b_j$};
\fill[darkgray,opacity=\opac](-13,7.16032)--(-19,-10)--(13,-10)--(13,-8.12199);

\draw (-2.98477,1.73526) circle (0.398036cm);
\draw (-3.79573,1.92658) circle (0.152036cm);
\draw (-3.54642,2.35073) circle (0.152036cm);
\draw (-3.40893,1.98457) circle (0.0939636cm);
\draw (-3.28427,2.19665) circle (0.152036cm);
\draw (-3.53358,1.7725) circle (0.152036cm);
\draw (-4.09523,2.38797) circle (0.398036cm);
\draw (-3.67107,2.13866) circle (0.0939636cm);
\draw (-3.50148,2.12715) circle (0.0760182cm);
\draw (-3.57852,1.99608) circle (0.0760182cm);
\fill[darkgray,opacity=\opac] (-2.98477,1.73526) circle (0.398036cm);
\fill[darkgray,opacity=\opac] (-3.79573,1.92658) circle (0.152036cm);
\fill[darkgray,opacity=\opac] (-3.54642,2.35073) circle (0.152036cm);
\fill[darkgray,opacity=\opac] (-3.40893,1.98457) circle (0.0939636cm);
\fill[darkgray,opacity=\opac] (-3.28427,2.19665) circle (0.152036cm);
\fill[darkgray,opacity=\opac] (-3.53358,1.7725) circle (0.152036cm);
\fill[darkgray,opacity=\opac] (-4.09523,2.38797) circle (0.398036cm);
\fill[darkgray,opacity=\opac] (-3.67107,2.13866) circle (0.0939636cm);
\fill[darkgray,opacity=\opac] (-3.50148,2.12715) circle (0.0760182cm);
\fill[darkgray,opacity=\opac] (-3.57852,1.99608) circle (0.0760182cm);
\end{scope}

\draw[->] (4,.25)--(4,1) node[midway,left]  {Scaling};
\end{tikzpicture}

    \caption{Example of a standard transformation.}
    \label{fig;standard}
\end{figure}
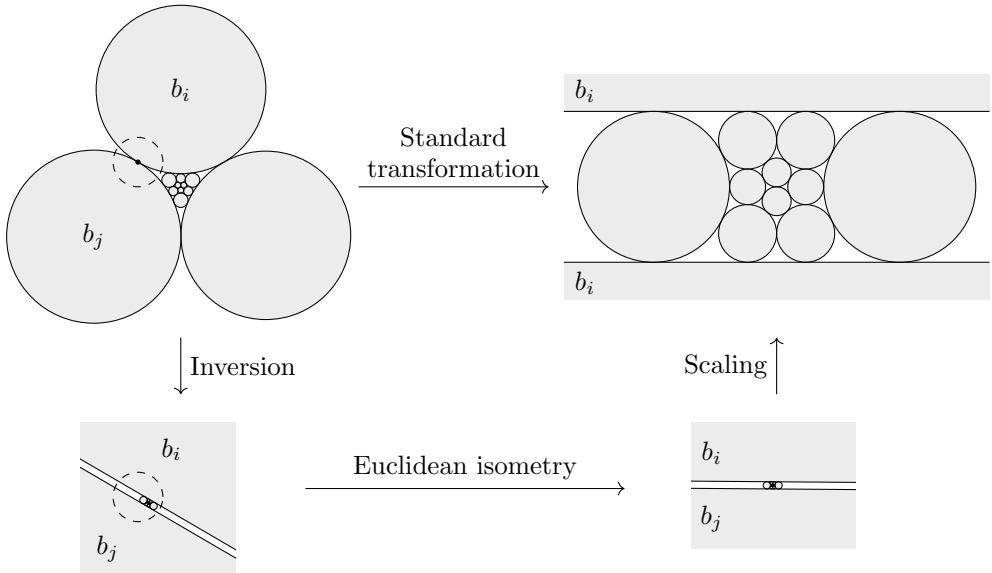

Two $d$-ball packings $\mathcal{B}$ and $\mathcal{B}'$ will be said to be \textit{Möbius congruent} if there exists $\mu\in M\ddot{o}b(d)$ such that $\mu:\mathcal{B}\mapsto\mathcal{B}'$. If in addition $\mu$ is a Euclidean isometry then we will say that $\mathcal{B}$ and $\mathcal{B}'$ are \textit{Euclidean congruent}.
\begin{rem}\label{generalization} Any $d$-ball packing is Möbius congruent to a $d$-ball packing formed by solid balls. 
\end{rem}

For any simple graph $G$ we define the space of equivalence classes under the action of the Möbius Group
\begin{eqnarray*}
\mathcal{M}^d(G):=\{d\text{-ball packings with tangency graph $G$}\}/_{M\ddot{o}b(d)}.
\end{eqnarray*}

We notice that a graph $G$ is \textit{$d$-ball packable} if and only if $\mathcal{M}^d(G)$ is not empty. We say that a $d$-ball packable graph is \textit{Möbius rigid} if all the $d$-ball packings with tangency graph $G$ are Möbius congruent, which is equivalent to say that  $\mathcal{M}^d(G)\simeq 1$. 
The advantage of a Möbius rigid graph $G$ is that all the properties which are preserved under the action of the Möbius group can be read in just one example of a disk packing $\mathcal{B}_G$.  A useful result to compute the space $\mathcal{M}^d(G)$ is the following:

\begin{lem}\label{congru} Let $\mathcal{B}_G$ and $\mathcal{B}'_G$ be two $d$-ball packings with same tangency graph $G$ and let $ij$ be an edge of $G$. Then $\mathcal{B}_G$ and  $\mathcal{B}'_G$ are Möbius congruent if and only if $[\mathcal{B}_G]^{i}_j$ and $[\mathcal{B}'_G]^{i}_j$ are Euclidean congruent.
\end{lem}
\begin{proof}
Let $\phi:\mathcal{B}_G\mapsto[\mathcal{B}_G]^{i}_j$ and $\psi:\mathcal{B}'_G\mapsto[\mathcal{B}'_G]^{i}_j$ be two standard transformations.\\

\noindent(Sufficiency) If there is a Euclidean isometry $\gamma:[\mathcal{B}_G]^{i}_j\mapsto[\mathcal{B}'_G]^{i}_j$
then $\psi^{-1}\circ\gamma\circ\phi$ defines a Möbius transformation mapping $\mathcal{B}_G$ to $\mathcal{B}'_G$. \\

\noindent(Necessity) Let us suppose that there is a Möbius transformation  $\mu:\mathcal{B}_G\mapsto\mathcal{B}'_G$. Then $\theta:=\psi\circ\mu\circ\phi^{-1}$ is a Möbius transformation mapping $[\mathcal{B}_G]^{i}_j$ to $[\mathcal{B}'_G]^{i}_j$ and leaving fixed the half-spaces $b_i$ and $b_j$. Therefore, $\theta$ is generated by inversions on $d$-balls which are simultaneously orthogonal to $b_i$ and $b_j$. A $d$-ball simultaneously orthogonal to two parallel half-spaces must be also a half-space. Therefore, $\theta$ can be expressed as a product of Euclidean reflections so $\theta$ is a Euclidean isometry.
\end{proof}
\vspace{1cm}
The family of $d$-ball packable graphs are fully characterized for $d=1,2$. Such characterization is still unknown nowadays when $d\ge3$.
Indeed, $d$-ball packable graphs are closely related to the $(d-1)$-ball packable graphs which can be made by $(d-1)$-balls of same size. It has been proved that the recognition of the tangency graphs of disk packings made by equal disks (and more generally disks with bounded ratio of diameters) turned out to be NP-hard \cite{breu}, see also \cite{hlineny}. However, many properties and constructions of $3$-ball packable graphs has been found, see \cite{maenoha}, \cite{mae2}, \cite{mae3}, \cite{kuperschr}, \cite{chenlabbe}, \cite{bezdekreid}.\\

From now on, we shall focus our attention to $d$-ball packings for $d=2,3$. In order to simplify the notation, we will call \textit{disks} (resp. \textit{balls}) the $2$-balls (resp. $3$-balls) and the collections of disks and balls will be denoted by $\mathcal{D}$ and $\mathcal{B}$  respectively.\\

Disk packable graphs were fully characterized in 1936 by Koebe  \cite{koebe}. The latter was rediscovered by Thurston by using some results of Andreev on hyperbolic $3$-polytopes. The well-known Koebe-Andreev-Thurston circle packing Theorem (KAT Theorem) states

\begin{thm}\label{KAT} A graph $G$ is disk packable if and only if $G$ is a simple planar graph.
Moreover, if $G$ is a triangulation of  $\mathbb{S}^2$ then $G$ is Möbius rigid.
\end{thm}

For a detailed survey on the applications of the KAT Theorem we refer the readers to a recent paper of Bowers  \cite{bowers}.

\subsection{Pyramidal disk systems}
The graph of a polyhedron is the graph made by its vertices and edges. Steinitz proved that the graphs of convex polyhedra are the $3$-connected simple planar graphs. We denote $\triangle$, $\lozenge$ and $\boxtimes$ the graphs of the tetrahedron, octahedron and a square pyramid respectively with the labeling given in Figure \ref{fig:polygraphs}.

\begin{figure}[H]
\centering

\begin{tikzpicture}[scale=1]

\begin{scope}[scale=1.5,xshift=-3.5cm]
\coordinate (OO) at (0,0);
\node (Vone) at (-0.866,0) [circle,fill=black,inner sep=0pt,minimum size=.15cm]  {};
\node (Vtwo) at (0.866,0) [circle,fill=black,inner sep=0pt,minimum size=.15cm]  {};
\node (Vthr) at (0,1.5) [circle,fill=black,inner sep=0pt,minimum size=.15cm]  {};
\node (VO) at (0,0.5) [circle,fill=black,inner sep=0pt,minimum size=.15cm]  {};
\draw[fill=black] (Vone) circle (0cm) node[anchor= north east] {$1$};
\draw[fill=black] (Vtwo) circle (0cm) node[anchor= north west] {$2$};
\draw[fill=black] (Vthr) circle (0cm) node[anchor= south,inner sep=5 pt] {$3$};
\draw[fill=black] (VO) circle (0cm) node[anchor=north,inner sep=5 pt] {$4$};
\draw (Vone) -- (Vtwo) -- (Vthr) -- (Vone);
\draw (Vone) -- (VO) -- (Vtwo);
\draw (VO) -- (Vthr); 
\node [below of= OO]  {(a)};

\end{scope}

\begin{scope}[scale=1.5]
\coordinate (OO) at (0,0);
\node (Vone) at (-0.866,0) [circle,fill=black,inner sep=0pt,minimum size=.15cm]  {};
\node (Vtwo) at (0.866,0) [circle,fill=black,inner sep=0pt,minimum size=.15cm]  {};
\node (Vthr) at (0,1.5) [circle,fill=black,inner sep=0pt,minimum size=.15cm]  {};
\node (VoneM) at (0.14,0.581) [circle,fill=black,inner sep=0pt,minimum size=.15cm]  {};
\node (VtwoM) at (-0.14,0.581) [circle,fill=black,inner sep=0pt,minimum size=.15cm]  {};
\node (VthrM) at (0,0.338) [circle,fill=black,inner sep=0pt,minimum size=.15cm]  {};
\draw[fill=black] (Vone) circle (0cm) node[anchor= north east] {$1$};
\draw[fill=black] (Vtwo) circle (0cm) node[anchor= north west] {$2$};
\draw[fill=black] (Vthr) circle (0cm) node[anchor= south,inner sep=5 pt] {$3$};
\draw[fill=black] (VoneM) circle (0cm) node[anchor= south west,inner sep=0 pt] {$-1$};
\draw[fill=black] (VtwoM) circle (0cm) node[anchor= south east,inner sep=1 pt] {$-2$};
\draw[fill=black] (VthrM) circle (0cm) node[anchor= north] {$-3$};
\draw (Vone) -- (Vtwo) -- (Vthr) -- (Vone);
\draw (VoneM) -- (VtwoM) -- (VthrM) -- (VoneM);
\draw (Vone) -- (VtwoM) -- (Vthr) -- (VoneM) -- (Vtwo) -- (VthrM) -- (Vone) ; 

\node [below of= OO] {(b)};

\end{scope}

\begin{scope}[xshift=5cm,yshift=1cm]
\coordinate (OO) at (0,-1);
\draw \SW rectangle \NE;
\draw \SW -- \NE;
\draw \NW -- \SE;

\draw[fill=black] (0,0) circle (.0cm) node[anchor= north,inner sep=6pt] {$x$};
\draw[fill=black] \NE circle (.0cm) node[anchor= south west] {$-1$};
\draw[fill=black] \NW circle (.0cm) node[anchor= south east] {$-2$};
\draw[fill=black] \SW circle (.0cm) node[anchor=north  east] {$1$};
\draw[fill=black] \SE circle (.0cm) node[anchor=north west] {$2$};

\node at (0,0) [circle,fill=black,inner sep=0pt,minimum size=.15cm]  {};
\node at \NE [circle,fill=black,inner sep=0pt,minimum size=.15cm]  {};
\node at \NW [circle,fill=black,inner sep=0pt,minimum size=.15cm]  {};
\node at \SW [circle,fill=black,inner sep=0pt,minimum size=.15cm]  {};
\node at \SE [circle,fill=black,inner sep=0pt,minimum size=.15cm]  {};
\node [below of= OO] {(c)};
\end{scope}

\end{tikzpicture}  
 \caption{(a) A planar embedding of $\triangle$; (b) a planar embedding of $\lozenge$; (c) a planar embedding of $\boxtimes$.}
 \label{fig:polygraphs}
\end{figure}
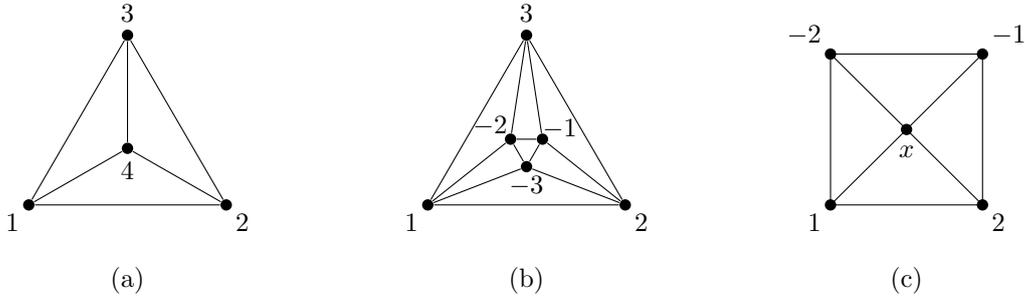 
Notice that $\boxtimes$ is isomorphic to the subgraph of $\lozenge$ obtained by deleting one vertex.
These three graphs are simple and planar and hence disk packable by the KAT theorem. We call a disk packing $\mathcal{D}_G$ \textit{tetrahedral}, \textit{octahedral} and \textit{pyramidal} if $G=\triangle,$ $\lozenge$, $\boxtimes$ respectively.

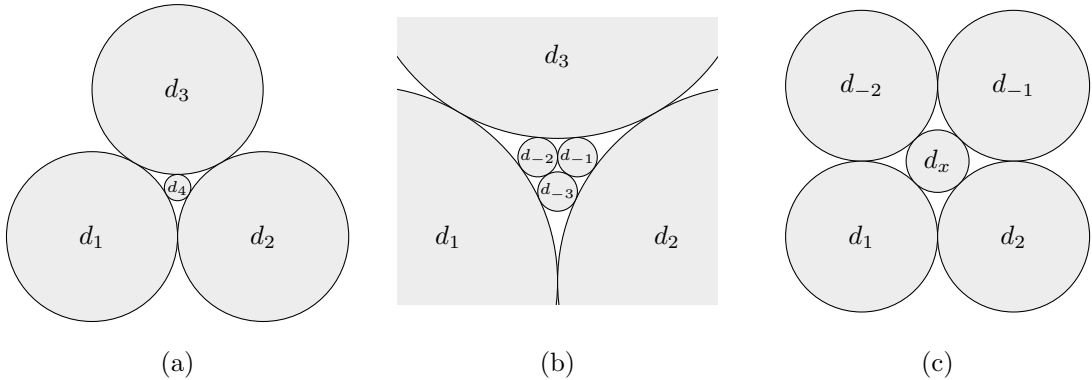
\begin{figure}[H]
\centering

\begin{tikzpicture}[scale=1]

\begin{scope}[xshift=-5cm,scale=1.3]

\draw[black,fill=mylgray] (-0.866,0) circle (0.866cm) node {$d_{1}$};
\draw[black,fill=mylgray] (0.866,0) circle (0.866cm) node {$d_{2}$};
\draw[black,fill=mylgray] (0,1.5)  circle (0.866cm) node {$d_{3}$};
\draw[black,fill=mylgray] (.0,.5)  circle (0.134cm) node {\tiny $d_{4}$};

\end{scope}

\begin{scope}[scale=3,yshift=-.2cm]
\clip (-.7,-.1) rectangle (.7,1.17);

\draw[black,fill=mylgray] (-0.866,0) circle (0.866cm) node at (-.48,.2) {$d_{1}$};
\draw[black,fill=mylgray] (0.866,0) circle (0.866cm) node at (.48,.2) {$d_{2}$};
\draw[black,fill=mylgray] (0,1.5)  circle (0.866cm) node at (0,1) {$d_{3}$};
\draw[black,fill=mylgray] (.087,.551)  circle (0.087cm) node {\tiny $d_{-1}$};
\draw[black,fill=mylgray] (-.087,.551)  circle (0.087cm) node {\tiny $d_{-2}$};
\draw[black,fill=mylgray] (0,.398979)  circle (0.087cm) node {\tiny $d_{-3}$};

\end{scope}

\begin{scope}[xshift=5cm,yshift=1cm]
\coordinate (OO) at (0,-1);

\draw[black,fill=mylgray] \NE circle (1cm) node {$d_{-1}$};
\draw[black,fill=mylgray] \NW circle (1cm) node {$d_{-2}$};
\draw[black,fill=mylgray] \SW circle (1cm) node {$d_{1}$};
\draw[black,fill=mylgray] \SE circle (1cm) node {$d_{2}$};
\draw[black,fill=mylgray] (0,0) circle (.414cm) node {$d_{x}$};

\end{scope}

\node at (-5,-1.7) {(a)};
\node at (0,-1.7) {(b)};
\node at (5,-1.7) {(c)};
\end{tikzpicture}  
   \caption{(a) A tetrahedral disk packing; (b) an octahedral disk packing; (c) a pyramidal disk packing.}
   \label{fig:diskpackpolys}
\end{figure}

Tetrahedral and octahedral disk packings have been well-studied. Since $\triangle$ and $\lozenge$ are triangulations of $\stw$,  $\triangle$ and $\lozenge$ are Möbius rigid. Many nice properties about the behaviour of the curvatures of the disks in tetrahedral and octahedral disk packings can be deduced from the Möbius rigidity, see  \cite{lagarias}. Unfortunately, pyramidal disk packings are not Möbius rigid as we show in the following.

\begin{prop}\label{prop;boxcongru}$\mathcal{M}^2(\boxtimes) \simeq \mathbb{R}$.
\end{prop}

\begin{proof}
Let $[\mathcal{D}_\boxtimes]^{-1}_x(\kappa_1)=\{d_x,d_1,d_2,d_{-1},d_{-2}\}$ be a standard disk packing where $d_2$ and $d_{-2}$ are two unit disks tangent to the half-spaces $d_{-1}=\{y\geq 1\}$, $d_x=\{y\leq-1\}$ and $d_1$ is a disk of curvature $\kappa_1\in\mathbb{R}$ tangent to $d_2$, $d_{-2}$ and $d_x$.

First of all, notice that $1<\kappa_1<4$. Indeed, when $\kappa_1<1$ (resp. $\kappa_1>4$) the disks $d_1$ and $d_{-1}$ (resp. $d_2$ and $d_{-2}$) intersect internally and when $\kappa_1=1$ (resp. $4$) $d_1$ and $d_{-1}$ (resp. $d_2$ and $d_{-2}$) would be tangent and the tangency graph would be other than $\boxtimes$, see Figure \ref{fig:pyrak}. We also notice that the collection of disk-packings $\{[\mathcal{D}_\boxtimes]^{-1}_x(\kappa_1)\}_{1<\kappa_1<4}$ are Euclidean non-congruent. Therefore, by Lemma \ref{congru}, they represent different equivalence classes in $\mathcal{M}^2(\boxtimes)$. Moreover, these are the only possible standard pyramidal disk packings. Hence, $\mathcal{M}^2(\boxtimes)$ is in bijection to the open interval $(1,4)$ which is homeomorphic to $\rone$.
\end{proof}

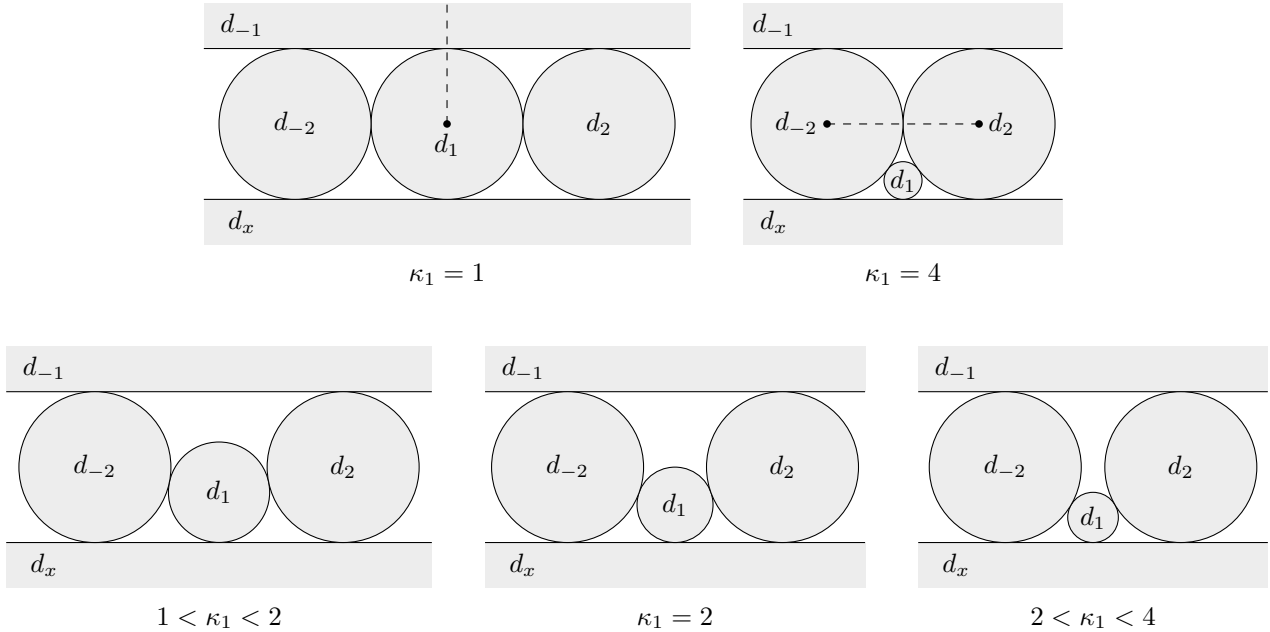
\begin{figure}[H]
\centering

\begin{tikzpicture}[scale=1]
\begin{scope}[scale=1]
\begin{scope}[xshift=-3cm,scale=1]
\fill [mylgray]  (-3.2,1) rectangle (3.2,1.6) node[black] at (-2.7,1.3) {$d_{-1}$};
\draw (-3.2,1) -- (3.2,1);
\fill [mylgray]  (-3.2,-1) rectangle (3.2,-1.6) node[black] at (-2.7,-1.3) {$d_x$};
\draw (-3.2,-1) -- (3.2,-1);

\draw[black,fill=mylgray] (0,0) circle (1cm) node[anchor= north] {$d_1$};
\draw[black,fill=mylgray] (2,0) circle (1cm) node {$d_2$};
\draw[black,fill=mylgray] (-2,0) circle (1cm) node {$d_{-2}$};

\node at (0,0) [circle,fill=black,inner sep=0pt,minimum size=.1cm]  {};

\draw[dashed] (0,0) -- (0,1.6);

\node at (0,-2) {$\kappa_1=1$};

\end{scope}

\begin{scope}[xshift=3cm,scale=1]

\fill [mylgray]  (-2.1,1) rectangle (2.1,1.6) node[black] at (-1.7,1.3) {$d_{-1}$};
\draw (-2.1,1) -- (2.1,1);
\fill [mylgray]  (-2.1,-1) rectangle (2.1,-1.6) node[black] at (-1.7,-1.3) {$d_{x}$};
\draw (-2.1,-1) -- (2.1,-1);

\draw[black,fill=mylgray] (0,-.75) circle (.25cm) node {$d_{1}$};
\draw[black,fill=mylgray] (1,0) circle (1cm) node[anchor=west] {$d_{2}$};
\draw[black,fill=mylgray] (-1,0) circle (1cm) node[anchor=east] {$d_{-2}$};
\node at (-1,0) [circle,fill=black,inner sep=0pt,minimum size=.1cm]  {};
\node at (1,0) [circle,fill=black,inner sep=0pt,minimum size=.1cm]  {};
\draw[dashed] (-1,0) -- (1,0);
\node at (0,-2) {$\kappa_1=4$};
\end{scope}
\end{scope}

\begin{scope}[yshift=-4.55cm]

\begin{scope}[xshift=-6cm,scale=1]
\fill [mylgray]  (-2.8,1) rectangle (2.8,1.6) node[black] at (-2.3,1.3) {$d_{-1}$};
\draw (-2.8,1) -- (2.8,1);
\fill [mylgray]  (-2.8,-1) rectangle (2.8,-1.6) node[black] at (-2.3,-1.3) {$d_x$};
\draw (-2.8,-1) -- (2.8,-1);
\draw[black,fill=mylgray] (0,-.333) circle (.666cm) node {$d_{1}$};
\draw[black,fill=mylgray] (1.633,0) circle (1cm) node {$d_{2}$};
\draw[black,fill=mylgray] (-1.633,0) circle (1cm) node {$d_{-2}$};
\node at (0,-2) {$1<\kappa_1<2$};
\end{scope}

\begin{scope}[scale=1]
\fill [mylgray]  (-2.5,1) rectangle (2.5,1.6) node[black] at (-2,1.3) {$d_{-1}$};
\draw (-2.5,1) -- (2.5,1);
\fill [mylgray]  (-2.5,-1) rectangle (2.5,-1.6) node[black] at (-2,-1.3) {$d_x$};
\draw (-2.5,-1) -- (2.5,-1);
\draw[black,fill=mylgray] (0,-.5) circle (.5cm) node {$d_{1}$};
\draw[black,fill=mylgray] (1.414,0) circle (1cm) node {$d_{2}$};
\draw[black,fill=mylgray] (-1.414,0) circle (1cm) node {$d_{-2}$};
\node at (0,-2) {$\kappa_1=2$};
\end{scope}

\begin{scope}[xshift=5.5cm,scale=1]
\fill [mylgray]  (-2.3,1) rectangle (2.3,1.6) node[black] at (-1.8,1.3) {$d_{-1}$};
\draw (-2.3,1) -- (2.3,1);
\fill [mylgray]  (-2.3,-1) rectangle (2.3,-1.6) node[black] at (-1.8,-1.3) {$d_x$};
\draw (-2.3,-1) -- (2.3,-1);
\draw[black,fill=mylgray] (0,-.667) circle (.333cm) node {$d_{1}$};
\draw[black,fill=mylgray] (1.155,0) circle (1cm) node {$d_{2}$};
\draw[black,fill=mylgray] (-1.155,0) circle (1cm) node {$d_{-2}$};
\node at (0,-2) {$2<\kappa_1<4$};
\end{scope}

\end{scope}

\end{tikzpicture}  
   \caption{Extreme cases with an extra edge (top figures) and the equivalence classes of $\mathcal{M}^2(\boxtimes)$ (bottom figures).}
   \label{fig:pyrak}
\end{figure}

Pyramidal disk packings are one of the main ingredients for  constructing the desired necklace. Since $\boxtimes$ is not Möbius rigid all the properties and the added structures must be carefully verified in each equivalence class of $\mathcal{M}^2(\boxtimes)$. To this end, in the same flavour as in the above proof, we define for every $i=1,2,-1,-2$, the \textit{standard curvatures} of a pyramidal disk packing $\mathcal{D}_\boxtimes$ the numbers $1<\kappa_i<4$ corresponding to the curvature of the disk $d_i$ in  $[\mathcal{D}_\boxtimes]^{-i}_x$. The standard curvatures can be used to identify the equivalence class of $\mathcal{D}_\boxtimes$ in $\mathcal{M}^2(\boxtimes)$. We define also the \textit{smaller standard curvature}  $\kappa:=\min\{\kappa_1,\kappa_2,\kappa_{-1},\kappa_{-2}\}$.  We notice that a pyramidal disk packing is a subset of an octahedral disk packing if and only if $\kappa=\kappa_1=\kappa_2=\kappa_{-1}=\kappa_{-2}=2$.\\

We define a \textit{pyramidal disk system} as the collection of disks $(\mathcal{D}_\boxtimes,d_1^*,d_2^*,d_t)$ formed by
\begin{itemize}
\item[•]$\mathcal{D}_\boxtimes=\{d_x,d_1,d_2,d_{-1},d_{-2}\}$: a disk-packing with tangency graph $\boxtimes$.
\item[•]The \textit{mirror disks} $d_1^*$ and $d_2^*$ where $d_{1}^*$ is the disk orthogonal to $d_2$, $d_{-2}$, $d_x$ and $d_1\subset d_1^*$; 
$d_{2}^*$  is the disk orthogonal to $d_1$, $d_{-1}$, $d_x$ and 
$d_2\subset d_2^*$.
\item[•]The \textit{tangency disk} $d_t$: the disk whose boundary passes through all the tangency points $d_1\cap d_2$, $d_1\cap d_{-2}$, $d_{-1}\cap d_2$ and $d_{-1}\cap d_{-2}$ and $d_x\subset d_t$.
\end{itemize}

\begin{lem}For any pyramidal disk packing the mirror disks and the tangency disk are well-defined.
\end{lem}
\begin{proof}
We first prove the Lemma for the standard $[\mathcal{D}_\boxtimes]^{-1}_x(\kappa_1)$ which appears in the Figure \ref{fig:pyramidalsystem}.\\

The orthogonality conditions of $d_1^*$ imply that the boundary of $d_1^*$ must be the circle with center $(0,-1)$ which passes through the tangency point $d_x\cap d_2$. The orientation of the interior is determined by the condition $d_1\subset d_1^*$.\\

For $d_2^*$, a disk orthogonal to $d_1$, $d_{-1}$ and $d_x$ must be a half-space with $y$-axis as the boundary. As before, the orientation of the interior comes from the condition $d_2\subset d_2^*$ which gives that $d_2^*$ is the half-space $\{x\geq 0\}$.\\

For $d_t$, by symmetry, the only circle passing through the tangency points $d_1\cap d_2$, $d_1\cap d_{-2}$ and $d_{-1}\cap d_2$ must passes through $d_{-1}\cap d_{-2}$. Again, the orientation is determined by the condition $d_x\subset d_t$.\\

It is clear that the previous arguments works for any standard $[\mathcal{D}_\boxtimes]^{-1}_x(\kappa_1)$ with $1<\kappa_1<4$. Since the conditions defining the mirror disks and the tangency disks are preserved under Möbius transformations, the Lemma is also true for all the pyramidal disk packings which are in the same class of 
$[\mathcal{D}_\boxtimes]^{-1}_x(\kappa_1)$ in $\mathcal{M}^2(\boxtimes)$, for any $1<\kappa_1<4$. As we showed in the proof of Proposition \ref{prop;boxcongru}, these class contains all the pyramidal disk packings.\\

\end{proof}

\begin{figure}[H]
\centering

\begin{tikzpicture}[scale=1.7]
\newcommand\XX{2.3}
\newcommand\YY{2}
\newcommand\YYmin{-2.5}

\fill [darkgray,opacity=\opac]  (-\XX,\YYmin) rectangle (\XX,\YY);
\fill [white]  (0,.667) circle (1.201cm) ;

\fill [darkgray,opacity=\opac]  (-\XX,1) rectangle (\XX,\YY);
\fill [darkgray,opacity=\opac]  (-\XX,-1) rectangle (\XX,\YYmin);
\fill [darkgray,opacity=\opac]  (0,\YYmin) rectangle (\XX,\YY);

\fill[darkgray,opacity=\opac](0,-.667) circle (.333cm);
\fill[darkgray,opacity=\opac] (1.155,0) circle (1cm);
\fill[darkgray,opacity=\opac] (-1.155,0) circle (1cm);
\fill[darkgray,opacity=\opac] (0,-1) circle (1.154cm);

\node at (-.289,-.5) [circle,fill=black,inner sep=0pt,minimum size=.1cm]  {};
\node at (.289,-.5) [circle,fill=black,inner sep=0pt,minimum size=.1cm]  {};
\node at (1.155,1) [circle,fill=black,inner sep=0pt,minimum size=.1cm]  {};
\node at (-1.155,1) [circle,fill=black,inner sep=0pt,minimum size=.1cm]  {};

\draw (0,-.667) circle (.333cm) node[anchor= north east, inner sep=2pt]  {$d_1$};
\draw (1.155,0) circle (1cm) node[anchor=west,inner sep=1.1cm] {$d_{2}$};
\draw (-\XX,1) -- (\XX,1) node at (-1.8,1.2) {$d_{-1}$};
\draw (-1.155,0) circle (1cm) node at (-1.8,0) {$d_{-2}$};
\draw (-\XX,-1) -- (\XX,-1) node at (-1.8,-1.2) {$d_{x}$};
\draw[dashed] (0,-1) circle (1.154cm) node at (-.7,-1.7) {$d_1^*$};
\draw[dashed] (0,\YY) -- (0,\YYmin) node at (.2,-2.35) {$d_{2}^*$};
\draw[dotted]  (0,.667) circle (1.201cm) node at (-.9,1.7) {$d_t$};

\end{tikzpicture}  
   \caption{The pyramidal disk system of a standard $[\mathcal{D}_\boxtimes]_x^{-1}(\kappa_1)$ with the center of $d_1$ lying on the $y$-axis. The boundaries of the mirror disks in dashed and for the tangency disk in dotted. The label of each disk lies on its interior.}
   \label{fig:pyramidalsystem}
\end{figure}
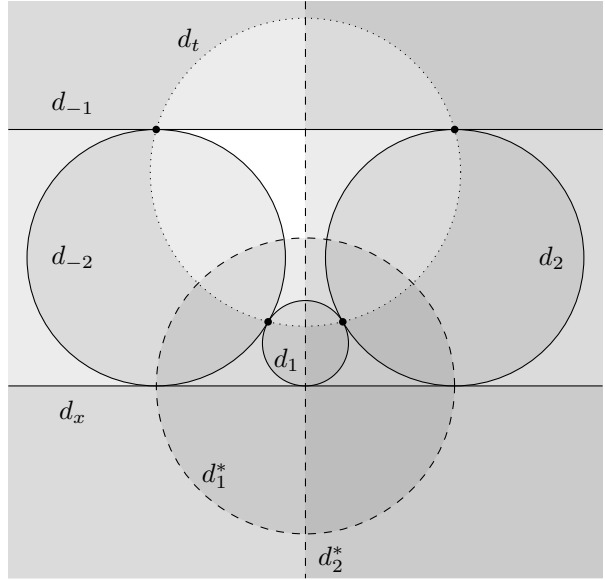 
\begin{table}[H]
$
\small
\begin{array}{l|l|cc|cccc}
\text{Disks}&\text{Curvature ($\delta$)}&\multicolumn{2}{c|}{\text{Center}}&\multicolumn{4}{c}{\text{Inversive coordinates}}\\

\hline
 d_x &0 \hspace{1.35cm} (1) & 0 & -1 & 0 & -1 & 1 & 1 \\
 d_1 & \kappa _1 & 0 & \frac{1}{\kappa _1} -1& 0 & 1-\kappa _1& -1 &\kappa_1-1 \\
 d_2 & 1 & \frac{2}{\sqrt{\kappa _1}} & 0 &
   \frac{2}{\sqrt{\kappa _1}} & 0 & \frac{2}{\kappa _1}-1 &
   \frac{2}{\kappa _1} \\
 d_{-1} & 0 \hspace{1.35cm} (1)  & 0 & 1 & 0 & 1 & 1 & 1 \\
 d_{-2} & 1 & \frac{-2}{\sqrt{\kappa _1}} & 0 & -2
   -\frac{2}{\sqrt{\kappa _1}} & 0 & \frac{2}{\kappa _1}-1 &
   \frac{2}{\kappa _1} \\
   \hline
 d_1^* & \frac{\sqrt{\kappa _1}}{2} & 0 & -1 & 0 &
   \frac{-\sqrt{\kappa _1}}{2} & \frac{-1}{\sqrt{\kappa _1}} &
   \frac{\kappa _1-2}{2 \sqrt{\kappa _1}} \\
 d_2^* & 0 \hspace{1.35cm} (0) & 1 & 0  & 1 & 0 & 0 & 0 \\
 \hline
 d_t & \frac{-\kappa _1}{\sqrt{\kappa_1^2+4}} & 0 & \frac{2}{\kappa _1} & 0 & \frac{-2}{\sqrt{\kappa _1^2+4}} &
   \frac{\kappa _1}{\sqrt{\kappa _1^2+4}} & 0 \\
\end{array}
$
\caption{Curvature, center and inversive coordinates of the disks of the pyramidal disk system in Figure \ref{fig:pyramidalsystem} in terms of the curvature of $d_1$. When a disk is a half-space the algebraic distance is given in brackets and the coordinates of the center are the coordinates of the normal vector.}
\label{tab:pyramid}
\end{table}

Given the inversive coordinates of Table \ref{tab:pyramid} we may compute the inversive products of the disks of a pyramidal disk system for each equivalence class of $\mathcal{M}^2(\boxtimes)$ in terms of the standard curvatures. 
\begin{lem}\label{lem:relations}
The following relations hold for every pyramidal disk system $(\mathcal{D}_\boxtimes,d_1^*,d_2^*,d_t)$ and for every $i=1,2$:
\begin{enumerate}[label=(\alph*),itemsep=5pt]
\item \label{lem:relations;ltok}$\langle d_i,d_{-i}\rangle=-1-2\kappa_i=-1-\frac{8}{\kappa_j}$ with $i\not=j$.
\item \label{lem:relations;ki}$\kappa_i=\kappa_{-i}$.
\item \label{lem:relations;k1k2}$\kappa_1\kappa_2=4$.
\item $-7<\langle d_i,d_{-i}\rangle<-1$.
\item $(1-\langle d_1,d_{-1}\rangle)(1-\langle d_2,d_{-2}\rangle)=16$.
\item \label{lem:relations;tangencydisk} $\partial d_t\subset d_1\cup d_2\cup d_{-1}\cup d_{-2}$. 
\item $d_1^*$, $d_2^*$ and $d_t$ are mutually orthogonal. 
\item \label{lem:relations;mirrors}
$\sigma_{d_i^*}(d_j)=
\left\lbrace\begin{array}{lll}
d_{-j}&\text{ if }i=|j|\\

d_j&\text{ otherwise} \\
\end{array}\right.$ for every $j\in\{1,2,-1,-2,t\}$.

\end{enumerate} 
\end{lem}
\begin{proof}
The relations can be obtained by simple calculations (combining equation (\ref{eq:invproduct}) and the inversive coordinates given in Table \ref{tab:pyramid}).
\end{proof}
The equalities \ref{lem:relations;ltok}, \ref{lem:relations;ki} and \ref{lem:relations;k1k2} tell us that a pyramidal disk packing has essentially two different standard curvatures $\kappa_1$ and $\kappa_2$ which are inversely proportional and the smaller standard curvature must verify $1<\kappa\le2$. We define the \textit{closest disjoint pair} of $\mathcal{D}_\boxtimes$ the disjoint pair $\{d_i,d_{-i}\}$ satisfying $\kappa=\kappa_i$, $i=1$ or $2$. The other disjoint pair will be called \textit{the farthest disjoint pair}. 
In the following we use the indices $\{d_c,d_{-c}\}$ and $\{d_f,d_{-f}\}$ with $\{c,f\}=\{1,2\}$ and  $c\not= f$ to denote the closest and the farthest disjoint pair of  $\mathcal{D}_\boxtimes$. By convention, we define $c=1$ and $f=2$ when $\kappa_1=\kappa_2$.

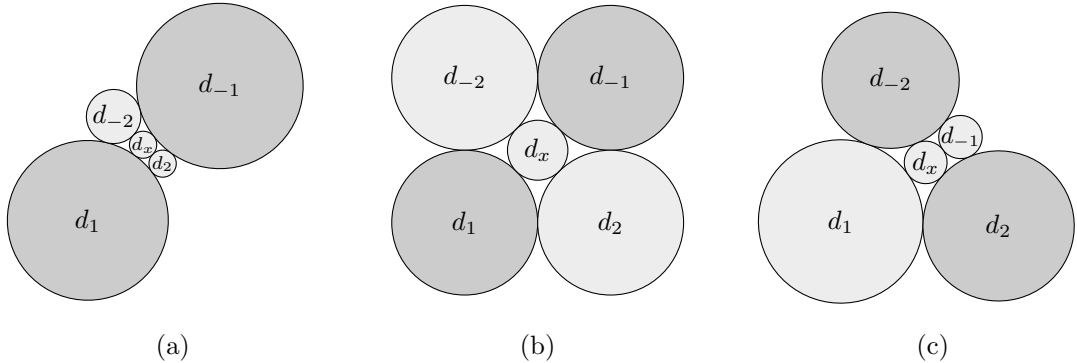
\begin{figure}[H]\label{fig:closefar}
\centering

\begin{tikzpicture}[scale=.8]
\begin{scope}[xshift=-6cm,scale=.45,rotate=220]
\draw[black,fill=mydgray] (4.38461, 0.389066) circle (2.94164cm) node {$d_1$};
\draw[black,fill=mylgray] (0.947749, 0.546203) circle (0.499cm) node {\scriptsize $d_2$};
\draw[black,fill=mydgray] (-2.49351, -0.3) circle (3.045cm) node {$d_{-1}$};
\draw[black,fill=mylgray] (1.20387, -1.934) circle (.997cm) node {$d_{-2}$};

\draw[black,fill=mylgray] (1.04725, -0.446) circle (.498691cm) node {\scriptsize $d_x$};
\end{scope}

\begin{scope}[scale=1.2,yshift=-0.2cm]
\coordinate (OO) at (0,-1);

\draw[black,fill=mydgray] \NE circle (1cm) node {$d_{-1}$};
\draw[black,fill=mylgray] \NW circle (1cm) node {$d_{-2}$};
\draw[black,fill=mydgray] \SW circle (1cm) node {$d_{1}$};
\draw[black,fill=mylgray] \SE circle (1cm) node {$d_{2}$};
\draw[black,fill=mylgray] (0,0) circle (.414cm) node {$d_{x}$};

\end{scope}

\begin{scope}[xshift=6cm,scale=.7,rotate=15]
\draw[black,fill=mylgray] (-1.927, -1.585) circle (1.934cm) node {$d_{1}$};
\draw[black,fill=mydgray] (1.629, -2.647) circle (1.777cm) node {$d_2$};
\draw[black,fill=mylgray] (1.308, -0.382) circle (.511cm) node {\small $d_{-1}$};
\draw[black,fill=mydgray] (0.071, 1.34) circle (1.611cm) node {$d_{-2}$};
\draw[black,fill=mylgray] (0.363, -0.751) circle (.503cm) node {$d_x$};
\end{scope}

\node at (-6,-3.5) {(a)};
\node at (0,-3.5) {(b)};
\node at (6.5,-3.5) {(c)};

\end{tikzpicture}  
   \caption{The closest disjoint pairs in darker gray in three different cases:\\
    (a) $\kappa=\kappa_1=1.33$; (b) $\kappa=\kappa_1=\kappa_2=2$; (c) $\kappa=\kappa_2=1.6$.}
\end{figure}

\subsection{The crossing ball system}
The main strategy for the proof of our main result is to construct a local ball packing around each crossing of the given diagram. We then stick together the local ball packings of two consecutive crossings. These local packings must take into account which piece of the curve goes over/under the other at each crossing of the link diagram. To this end, we may introduce {\em crossing ball systems} which are made from the blowing-up of a pyramidal disk system. There will be an over/under choice which is determined by a signed parameter $\epsilon\in\{+,-\}$.
\begin{rem}\label{rem:blowinv}The blowing up operation preserves the inversive product. 
\end{rem}
A \textit{pyramidal ball packing} $\mathcal{B}_\boxtimes=\{b_x,b_1,b_2,b_{-1},b_{-2}\}$ is a ball packing obtained by blowing up a pyramidal disk packing. We define equivalently the closest and farthest disjoint pairs as in the planar case. Let $(\mathcal{D}_\boxtimes,d^*_1,d_2^*,d_t)$ be a pyramidal disk system. We define for every $\epsilon\in\{+,-\}$, a \textit{crossing ball system} $(\mathcal{B}_\boxtimes,b_1^*;b_2^*,b_t,b_{\epsilon3},b_{\epsilon3}')$ as the arrangement of balls formed by:
\begin{itemize}
\item The \textit{pyramid ball packing} $\mathcal{B}_\boxtimes$ : the blowing up of $\mathcal{D}_\boxtimes$.
\item The \textit{mirror balls} $b_1^*$ and $b_2^*$: the blowing up of the mirror disks $d_1^*$ and $d_2^*$ respectively.
\item The \textit{tangency ball} $b_t$: the blowing up of the tangency disk $d_t$.
\item The \textit{bridge balls} $b_{\epsilon3}$ and $b_{\epsilon3}'$ where:
\begin{itemize}
\item[(i)]$b_{\epsilon3} $ is the unique ball externally tangent to $b_c$, $b_f$, $b_x$, internally tangent to $b_c^*$ and contained in the half-space $\{\epsilon z\geq 0\}$, where $\{b_c,b_{-c}\}$ and $\{b_f,b_{-f}\}$ denotes the closest and the farthest pair of $\mathcal{B}_\boxtimes$.  
\item[(ii)]$b_{\epsilon3}'$ is the ball obtained by the inversion of $b_{\epsilon3}$ on the mirror ball $b_c^*$.
\end{itemize}

\end{itemize}

We also define the \textit{crossing region} $\mathcal{R}$ of a crossing ball system as $\mathcal{R}=\left(\underset{b\in\mathcal{B}_\boxtimes}{\bigcap} -b\right) \cap b_t.$

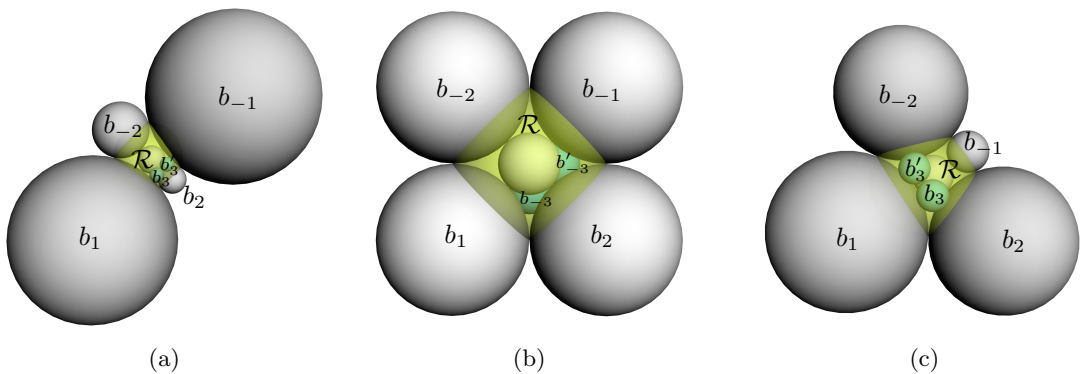
\begin{figure}[H]
\centering

\begin{tikzpicture}[scale=.8]
\begin{scope}[xshift=-6cm,yshift=-.3cm,scale=.45,rotate=0]
\node at (0,0) {\includegraphics[scale=.35]{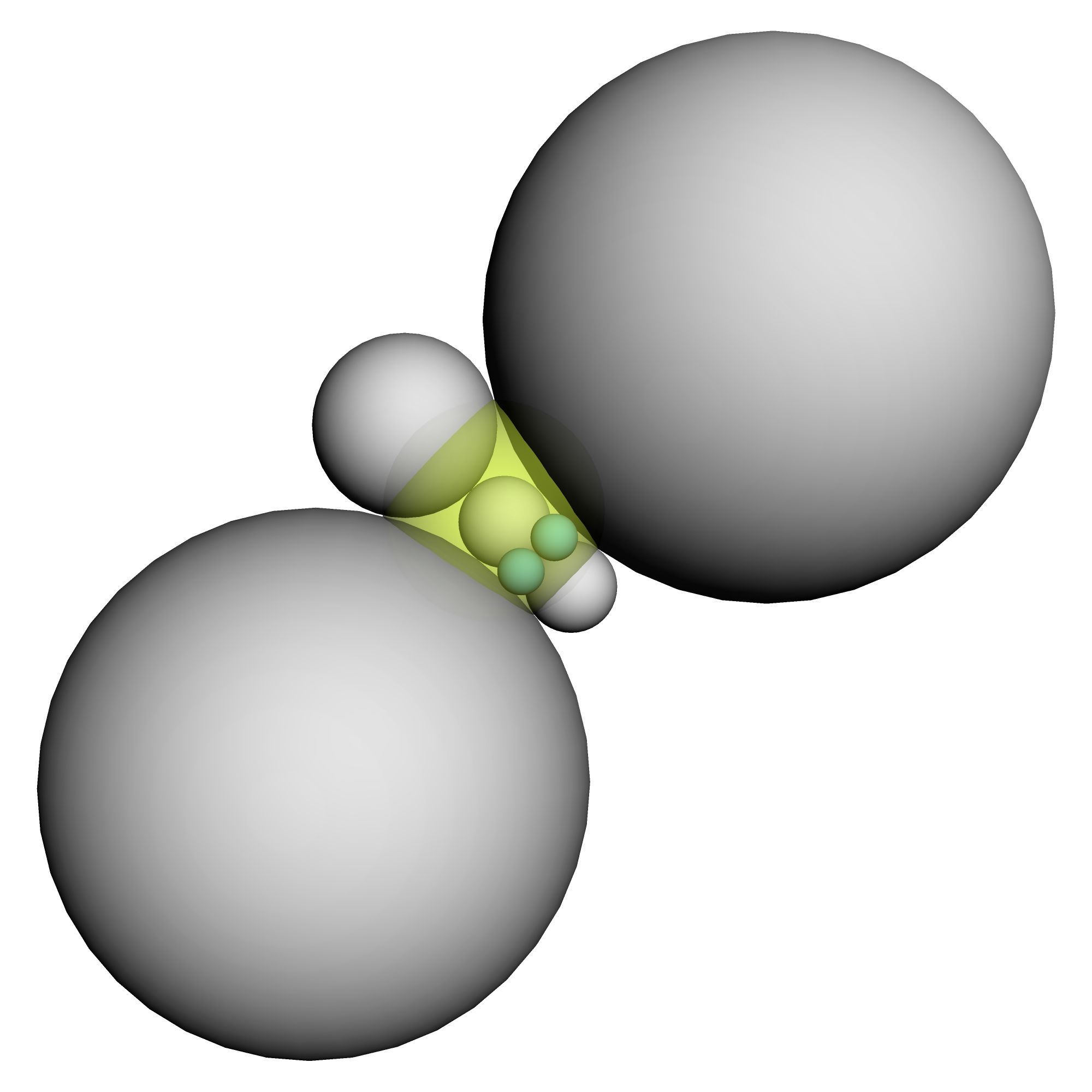}};
\node at (-2.7,-2.6) {$b_{1}$};
\node at (-1.5,1.5) {$b_{-2}$};
\node at (2.7,2.5) {$b_{-1}$};
\node at (1.1,-1.1) {$b_{2}$};
\node at (-.8,.3) {$\mathcal{R}$};
\node at (-.1,-.4) {\scriptsize $b_{3}$};
\node at (.2,0.1) {\scriptsize $b_3'$};

\end{scope}

\begin{scope}[scale=1.2,yshift=-0.2cm]
\coordinate (OO) at (0,-1);

\node at (0,0) {\includegraphics[scale=.35]{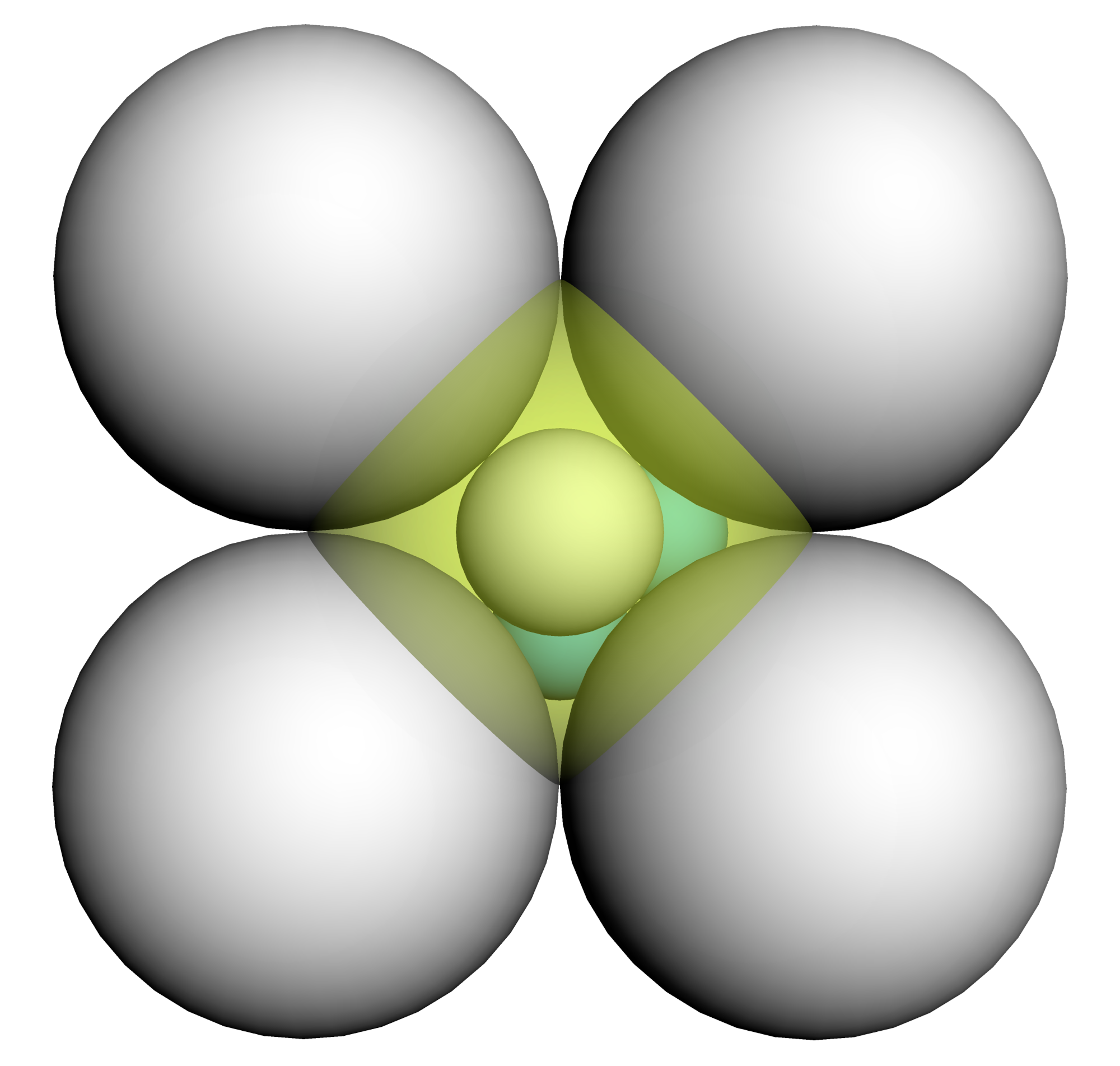}};
\node at \NE {$b_{-1}$};
\node at \NW {$b_{-2}$};
\node at \SW {$b_{1}$};
\node at \SE {$b_{2}$};
\node at (0,0.55) {$\mathcal R$};
\node at (.6,0) {\scriptsize $b_{-3}'$};
\node at (0.1,-.5) {\scriptsize $b_{-3}$};

\end{scope}

\begin{scope}[xshift=6.5cm,yshift=-.4cm,scale=.7]
\node at (0,0) {\includegraphics[scale=.37]{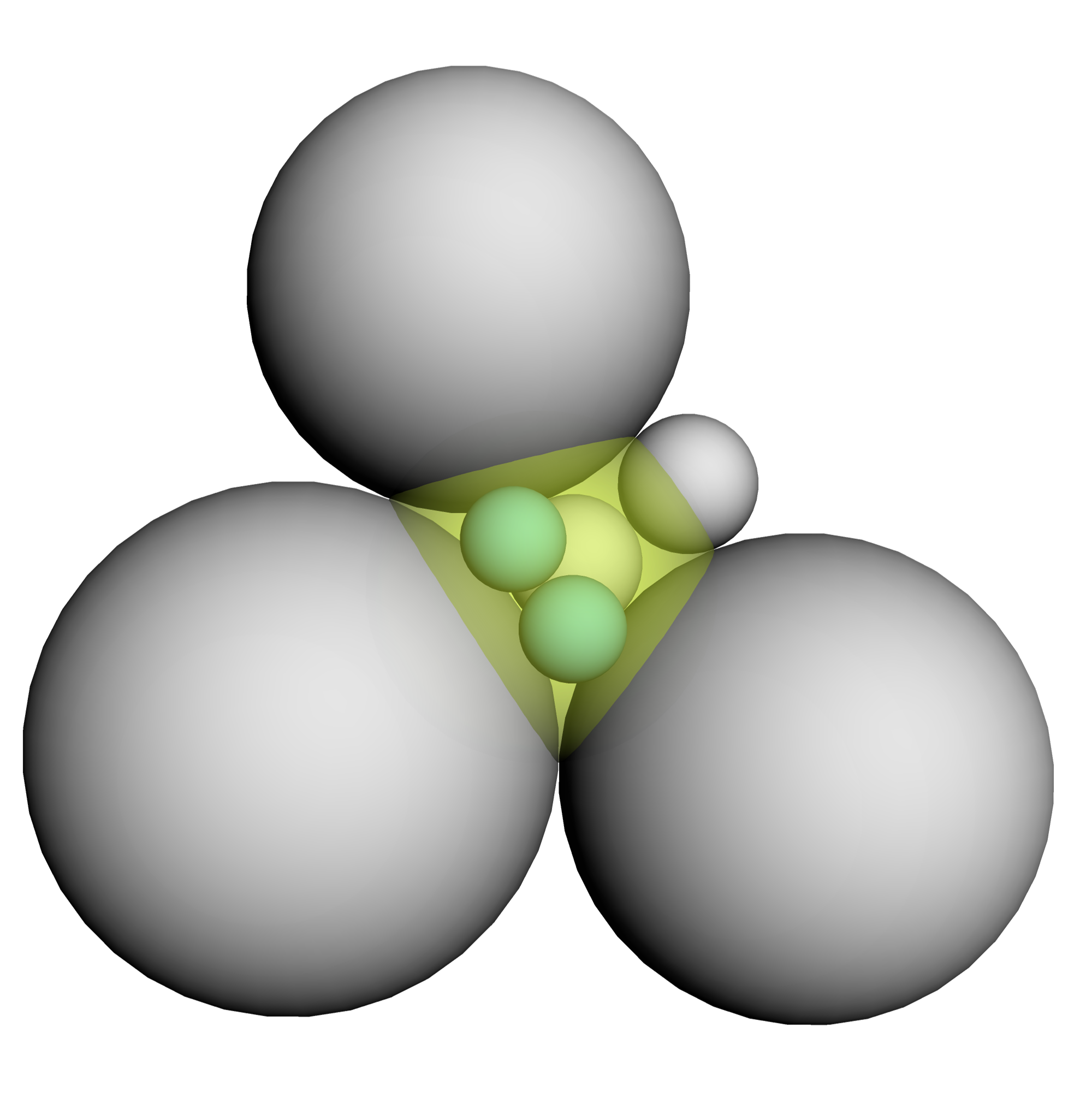}};
\node at (-1.8,-1.55) {$b_{1}$};
\node at (2.1,-1.6) {$b_{2}$};
\node at (1.4,.7) {$b_{-1}$};
\node at (-.6,1.8) {$b_{-2}$};

\node at (0.6,0.1) {$\mathcal{R}$};
\node at (-.2,0.1) {\small $b_{3}'$};
\node at (0.25,-.5) {\small $b_{3}$};
\end{scope}

\node at (-6,-3.5) {\small (a)};
\node at (0,-3.5) {\small(b)};
\node at (6.5,-3.5) {\small(c)};

\end{tikzpicture}  
   \caption{The pyramidal ball packing, bridge balls (blue) and crossing region (yellow) of three crossing ball systems seen from above:
    (a) $\kappa=\kappa_1=1.33$, $\epsilon=+$; (b) $\kappa=\kappa_1=\kappa_2=2$, $\epsilon=-$; (c) $\kappa=\kappa_2=1.6$, $\epsilon=+$.}
    \label{fig:xsystem}
\end{figure} 

\newpage
\begin{lem}\label{region} Let $(\mathcal{B}_\boxtimes,b_1^*,b_2^*,b_t,b_{\epsilon3},b_{\epsilon3}')$ be a crossing ball system. The bridge balls $b_{\epsilon3}$ and $b_{\epsilon3}'$ are well-defined for every $\epsilon\in\{+,-\}$. Moreover, they are externally tangent and both are contained in the crossing region $\mathcal{R}$. 
\end{lem}

\begin{proof}

 Consider the collection of balls 
\begin{eqnarray*}
\mathcal{B}=\{b_x,b_c,b_f,b_c^*,b_{\epsilon z}\}
\end{eqnarray*}
 where $\{b_c,b_{-c}\}$ and $\{b_f,b_{-f}\}$ are the closest and the farthest disjoint pair of $\mathcal{B}_\boxtimes$ and $b_{\epsilon z}$ is the half-space $\{\epsilon z\geq0\}$.  
 Since the inversive product is preserved by the blowing up operation, we can compute the Gramian of $\mathcal{B}$ by using the inversive coordinates given in the Table \ref{tab:pyramid} in terms of the smaller standard curvature.
\begin{eqnarray*}\label{gram}
\Gram(\mathcal{B})=
\begin{pmatrix}
1&-1&-1&0&0 \\
-1&1 &-1&\sqrt{\kappa}&0 \\
-1&-1&1 &0&0\\
0&\sqrt{\kappa}&0&1 &0\\
0&0&0&0&1
\end{pmatrix}&\text{ and }&\Gram(\mathcal{B})^{-1}=\frac{1}{2}
\begin{pmatrix}
\frac{\kappa }{2} & -1 & \frac{\kappa }{2}-1 & \sqrt{\kappa } & 0 \\
 -1 & 0 & -1 & 0 & 0 \\
 \frac{\kappa}{2} -1& -1 & \frac{\kappa }{2} & \sqrt{\kappa } & 0 \\
 \sqrt{\kappa } & 0 & \sqrt{\kappa } & 2 & 0 \\
 0 & 0 & 0 & 0 & 2 \\
\end{pmatrix}
\end{eqnarray*}
Since $\det(\Gram(\mathcal{B}))=4\not=0$ the Lorentzian vectors of $\mathcal{B}$ form a basis of $\mathbb{R}^{4,1}$. In order to show that the bridge balls are well-defined we compute the polyspherical coordinates of  $b_{\epsilon3}$ with respect to $\mathcal{B}$ using the definition of $b_{\epsilon3}$ and equation (\ref{eq:lprod}):
\begin{eqnarray}\label{eq:polyb3}
P_{\mathcal{B}}(b_{\epsilon3})=
\begin{pmatrix}
-1 \\
-1 \\
-1\\
1\\
\lambda_{z,3}
\end{pmatrix} \text{ with } \lambda_{z,3}\ge 1.
\end{eqnarray}
By using equation (\ref{eq:polyprod}) we can normalize to get  $\lambda_{z,3}=\sqrt{3+2\sqrt{\kappa}-\kappa}$. 
It can be checked that  $\lambda_{z,3}>1$ for every $1<\kappa\le 2$.
The latter implies the existence and the uniqueness of $b_{\epsilon3}$ and hence for $b_{\epsilon3}':=\sigma_{b_{c}^*}(b_3)$ for every pyramidal ball packing. Moreover, 
\begin{eqnarray*}
\langle b_{\epsilon3},b_{\epsilon3}'\rangle&=&\langle b_{\epsilon3},\sigma_{b_c^*}(b_{\epsilon3})\rangle\\
&=&\langle b_{\epsilon3},b_{\epsilon3}-2\langle b_{\epsilon3},b_c^*\rangle b_c^*\rangle\hspace{2cm}\text{by (\ref{eq:lorenref})}\\
&=&1-2\langle b_{\epsilon3},b_c^*\rangle^2\\
&=&-1
\end{eqnarray*}
so $b_{\epsilon3}$ and $b_{\epsilon3}'$ are externally tangent.\\
A ball $b$ is contained in the crossing region of the crossing ball system $(\mathcal{B}_\boxtimes,b_1^*,b_2^*,b_t,b_{\epsilon3},b_{\epsilon3}')$ if and only if 
\begin{equation}
\langle b_i,b\rangle\le-1\text{ for every }b_i\in\{b_x,b_c,b_f,b_{-c},b_{-f}\}\text{ and }\langle b_t,b\rangle\ge 1.
\end{equation}
By combining the invariance of the inversive product under inversions, the intersection angles of the mirrors and the other balls given in Lemma \ref{lem:relations} \ref{lem:relations;mirrors}  and  the tangency conditions in the definition of $b_{\epsilon3}$ we obtain
\begin{align*}
\langle b_{x}, b_{\epsilon3}'\rangle=&\langle \sigma_{b_c^*}(b_x),\sigma_{b_c^*}(b_{\epsilon3}')\rangle=\langle b_x, b_{\epsilon3}\rangle=-1,\\
\langle b_{-c}, b_{\epsilon3}'\rangle=&\langle \sigma_{b_c^*}(b_{-c}),\sigma_{b_c^*}(b_{\epsilon3}')\rangle=\langle b_c, b_{\epsilon3} \rangle=-1\quad\text{  and}\\
\langle b_{f}, b_{\epsilon3}'\rangle=&\langle \sigma_{b_c^*}(b_{f}),\sigma_{b_c^*}(b_{\epsilon3}')\rangle=\langle b_f, b_{\epsilon3}\rangle=-1.
\end{align*}

\noindent For the rest of inversive products we use Lemma \ref{lem:relations} \ref{lem:relations;mirrors}, equation (\ref{eq:polyprod}) and the inversive coordinates of Table \ref{tab:pyramid} :
\begin{align*}
\langle b_{c},b_{\epsilon3}'\rangle=\langle \sigma_{b_c^*}(b_{c}),\sigma_{b_c^*}(b_{\epsilon3}')\rangle=&\langle b_{-c},b_{\epsilon3}\rangle\\
=&P_{\mathcal{B}}(b_{-c})^T \Gram(\mathcal{B})^{-1}P_{\mathcal{B}}(b_{\epsilon3})\\
=&\begin{pmatrix}
-1&-2\kappa+1&-1&-\sqrt{\kappa}&0
\end{pmatrix}
\Gram(\mathcal{B})^{-1}
\begin{pmatrix}
-1\\
-1\\
-1\\
1\\
\sqrt{3+2\sqrt{\kappa}-\kappa}
\end{pmatrix}\\
=&-1-2\sqrt{\kappa}<-1\quad\text{ for }1<\kappa\le 2.
\end{align*}
By the same procedure we obtain:
\begin{align*}
\langle b_{-f},b_{\epsilon3}'\rangle=&\langle \sigma_{b_c^*}(b_{-f}),\sigma_{b_c^*}(b_{\epsilon3}')\rangle=\langle b_{-f},b_{\epsilon3}\rangle=3-\frac{4}{\sqrt{\kappa}}-\frac{8}{\kappa}<-1&\text{ for }1<\kappa\le 2\quad\text{ and}\\
\langle b_t,b_{\epsilon3}' \rangle=&\langle \sigma_{b_c^*}(b_{t}),\sigma_{b_c^*}(b_{\epsilon3}')\rangle=\langle  b_{t},b_{\epsilon3}\rangle=\frac{2+2\sqrt{\kappa}-\kappa}{\sqrt{4+\kappa^2}}\geq1&\text{ for }1<\kappa\le 2.
\end{align*}

\end{proof}

\section{The proof of the main theorem.}\label{sec:proof}
\subsection{From links to disk packable graphs.} 

The \textit{medial graph} of a planar graph $G$, denoted $med(G)$, is constructed by placing one vertex on each edge of $G$ and joining two vertices if the corresponding edges are consecutive on a face of $G$. Medial graphs are $4$-regular planar graphs but may have loops or multiple edges, see Figure \ref{fig:graphs} (c). We define the \textit{simplified medial graph} of $G$, denoted $\underline{med}(G)$, the simple planar graph obtained from $med(G)$ by deleting loops and multiple edges. Every link diagram $\LL$ with at least one crossing leads to a $4$-regular planar graph $G_\LL$ where the vertices are the crossings and the edges are the arcs joining the crossings. Again, this graph is not simple in general, see Figure \ref{fig:graphs} (b). We define the \textit{pyramidal patchwork}
of $\LL$ the simple planar graph given by the simultaneous drawing of $G_\LL\cup \underline{med}(G_\LL)$ and we denote this graph $\bigotimes(\LL)=(V_\otimes,E_\otimes)$. The set of vertices can be divided in two sets $V_\otimes=V_\times\cup V_m$ where $V_\times$ is the set of vertices of $G_{\LL}$ and $V_m$ is the set of vertices of $\underline{med}(G_\LL)$. We call the vertices of $V_\times$ the \textit{crossing vertices}.

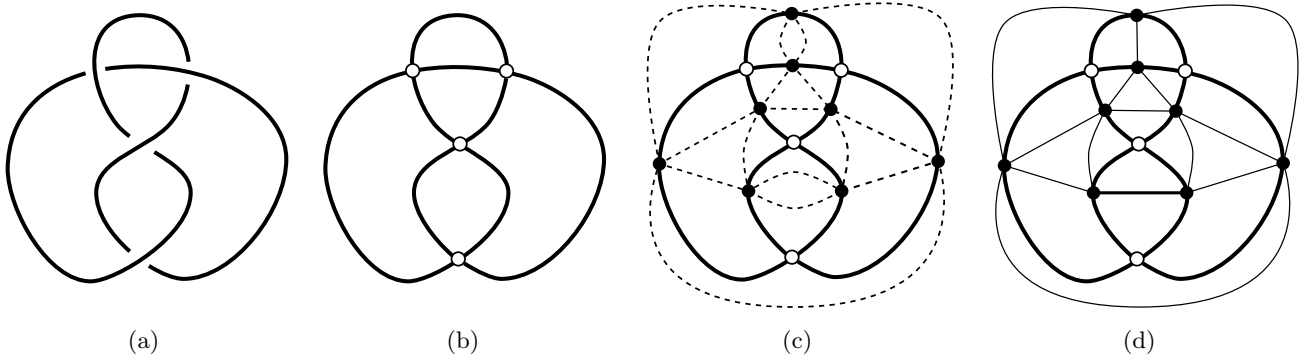
\begin{figure}[H]
  \centering
  \begin{tikzpicture}
    \node at (0,0) {\includegraphics[width=1\textwidth]{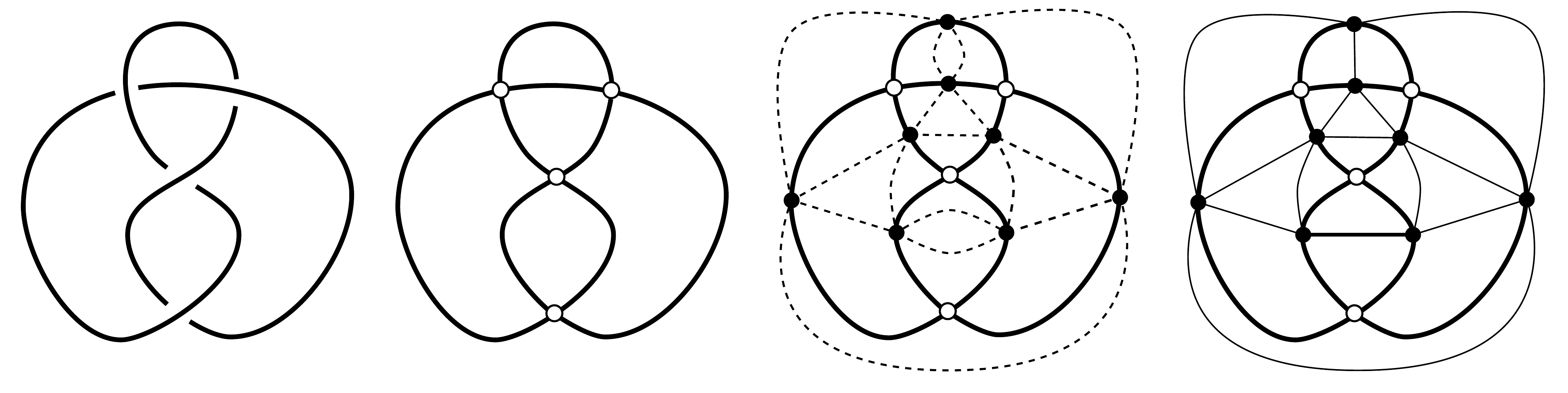}};
    \node at (-6.7,-2.4) {\small (a)};
    \node at (-2.5,-2.4) {\small (b)};
    \node at (1.9,-2.4) {\small (c)};
    \node at (6.4,-2.4) {\small (d)};
    
    \end{tikzpicture}
    \caption{(a) A diagram $\LL$ of the Figure-eight knot; (b) $G_\LL$; (c) $G_\LL$ with its medial in dashed; (d) the pyramidal patchwork of $\LL$.}
  \label{fig:graphs}
\end{figure}

Since $G_\LL$ is $4$-regular, the subgraph of $\bigotimes(\LL)$ induced by a crossing vertex and its 4 neighbours is a pyramidal graph. Therefore, the pyramidal patchwork of $\mathcal{L}$ can be obtained as the union of $n$ pyramidal graphs where $n$ is the number of crossings of $\mathcal{L}$. Moreover, we have
\begin{equation}
|V_\otimes|=|V_\times|+|V_m|\\
=n+\frac{1}{2}(4n)\\
=3n.
\end{equation}

\vspace{2cm}
We now have all the ingredients to proceed with the proof of Theorem \ref{theorem}.
\subsection{Proof of Theorem \ref{theorem}} For trivial links the theorem is trivially true. Let $L$ be a non-trivial link and let $\LL$ be a minimal crossing diagram of $L$. By the KAT theorem, there is a disk packing $\mathcal{D}_{\bigotimes(\LL)}$ with tangency graph $\bigotimes(\LL)=(V_\times\cup V_m,E_\otimes)$. 
Let $\mathcal{B}_{\bigotimes(\LL)}$ be the blowing up of $\mathcal{D}_{\bigotimes(\LL)}$. For every crossing vertex $x\in V_\times$, $\mathcal{D}_{\bigotimes(\LL)}$ 
admits a pyramidal disk system with disk packing $\mathcal{D}_\boxtimes(x)$ and therefore $\mathcal{B}_{\bigotimes(\LL)}$ admits a crossing ball system with pyramidal ball packing $\mathcal{B}_\boxtimes(x)$ and bridge balls $\{b_{\epsilon_x3},b_{\epsilon_x3}'\}$. We notice that
\begin{eqnarray*}
\mathcal{D}_{\bigotimes(\LL)}=\bigcup_{x\in V_\times}\mathcal{D}_\boxtimes(x)&\text{and}&\mathcal{B}_{\bigotimes(\LL)}=\bigcup_{x\in V_\times}\mathcal{B}_\boxtimes(x).
\end{eqnarray*}

We choose $\epsilon_x$ such that the thread of the chain made by the balls $(b_c,b_{\epsilon3},b_{\epsilon3}',b_{-c})$ is over/under the thread of the chain $(b_f,b_x,b_{-f})$ according to the diagram $\LL$. Let $\mathcal{B}_{\bigwedge(\LL)}$ be the collection of all the bridge balls with the appropriate signs with respect to $\mathcal{L}$ for each crossing vertex. Let
 $\mathcal{B}_\LL$ be the ball collection $\mathcal{B}_{\bigotimes(\LL)}\cup\mathcal{B}_{\bigwedge(\LL)}$. If $\mathcal{B}_\LL$ were a packing then its carrier would contain a polygonal link ambient isotopic to $L$ (by construction). Moreover, the number of balls $|\mathcal{B}_\LL|=|\mathcal{B}_{\bigotimes(\LL)}|+|\mathcal{B}_{\bigwedge(\LL)}|=3\mathsf{cr}(\LL)+2\mathsf{cr}(\LL)=5\mathsf{cr}(L)$ since $\mathcal{L}$ is a minimal crossing diagram.\\
 
We need thus to prove that $\mathcal{B}_\LL$ is a packing. To this end, it is enough to show the following three claims:
\begin{enumerate}
\item $\mathcal{B}_{\bigotimes(\mathcal{L})}$ is a packing.
\item Every ball of $\mathcal{B}_{\bigotimes(\mathcal{L})}$ is at most tangent to every ball of $\mathcal{B}_{\bigwedge(\LL)}$. 
\item $\mathcal{B}_{\bigwedge(\mathcal{L})}$ is a packing.
\end{enumerate}

\bigskip
Claim $(1)]$ This is obtained directly by the Remark \ref{rem:blowinv}.\\

Claim $(2)]$ Let $x$ be a crossing vertex with corresponding pyramidal disk system $(\mathcal{D}_\boxtimes,d_1^*,d_2^*,d_t)$, crossing ball system $(\mathcal{B}_\boxtimes,b_1^*,b_2^*,b_t,b_{\epsilon_x3},b_{\epsilon_x3}')$ and crossing region $\mathcal{R}$. Since $\mathcal{D}_{\bigotimes(\mathcal{L})}$ is a packing then, as a consequence of Lemma \ref{lem:relations} \ref{lem:relations;tangencydisk}, any disk $d\in \mathcal{D}_{\bigotimes(\mathcal{L})}\setminus\mathcal{D}_\boxtimes$ must be disjoint to $d_t$.
Therefore, the corresponding ball $b\in \mathcal{B}_{\bigotimes(\mathcal{L})}\setminus\mathcal{B}_\boxtimes$ must be disjoint to $b_t$ and thus, $b$ is disjoint to $\mathcal{R}$ that contains the bridge balls $b_{\epsilon_x3}$ and $b_{\epsilon_x3}'$ by Lemma \ref{region}. Hence, $b_{\epsilon_x3}$ and $b_{\epsilon_x3}'$ 
 are disjoint to every ball of $\mathcal{B}_{\bigotimes(\mathcal{L})}\setminus\mathcal{B}_\boxtimes$. In the other hand, Lemma \ref{region} ensures that $b_{\epsilon_x3}$ and $b_{\epsilon_x3}'$ are at most tangent to the balls of $\mathcal{B}_\boxtimes$.\\

Claim $(3)]$ We first notice that,  by Lemma \ref{region}, the bridge balls of a crossing system are externally tangent. It remains to show that the bridge balls of different crossing systems have disjoint interiors. Since $L$ is non-trivial, $\bigotimes(\mathcal{L})$ has at least two crossing vertices. Let $x$ and $x'$ be two different crossing vertices of $\bigotimes(\mathcal{L})$ and let
$\mathcal{D}_\boxtimes=\{d_x,d_1,d_2,d_{-1},d_{-2}\}$
and $\mathcal{D}_\boxtimes'=\{d_{x'},d_{1'},d_{2'},d_{-1'},d_{-2'}\}$ be the corresponding pyramidal disk packings in $\mathcal{D}_{\bigotimes(\LL)}$. Let $n$ be the number of disks in common of $\mathcal{D}_\boxtimes$ and $\mathcal{D}_\boxtimes'$ and let $\mathcal{R}$ and $\mathcal{R}'$ be their respective crossing regions. We end the proof by showing that in each of the five cases ($n=0,1,2,3,4$), $\mathcal{R}$ and $\mathcal{R}'$ have disjoint interiors. This implies, by Lemma \ref{region}, that the bridge balls of the crossing ball systems of $x$ and $x'$ are at most tangent.\\

If needed we may relabel $\mathcal{D}_\boxtimes'$ in order to work with the labelling of the graphs showed on the left.\\
\vspace{1cm}

\noindent
$n=0$\hspace{0.7cm}
\begin{minipage}[c]{.25\textwidth}
\begin{tikzpicture}[scale=.45]
\begin{scope}[xshift=-2.5cm]
\draw \SW rectangle \NE;
\draw[ultra thick] (-1.4,-1.4) -- (1.4,1.4);
\draw[ultra thick] (-1.4,1.4) -- (1.4,-1.4);

\draw[fill=black] (0,0) circle (.0cm) node[anchor= south] {\small $x$};
\draw[fill=black] \NE circle (.0cm) node[anchor= west] {$1$};
\draw[fill=black] \SE circle (.0cm) node[anchor= west] {$2$};
\draw[fill=black] \SW circle (.0cm) node[anchor= east] {$-1$};
\draw[fill=black] \NW circle (.0cm) node[anchor= east] {$-2$};

\node at (0,0) [circle,draw=black,fill=white,inner sep=0pt,minimum size=.15cm]  {};
\node at \NE [circle,fill=black,inner sep=0pt,minimum size=.15cm]  {};
\node at \NW [circle,fill=black,inner sep=0pt,minimum size=.15cm]  {};
\node at \SW [circle,fill=black,inner sep=0pt,minimum size=.15cm]  {};
\node at \SE [circle,fill=black,inner sep=0pt,minimum size=.15cm]  {};

\end{scope}

\begin{scope}[xshift=2.5cm]
\draw \SW rectangle \NE;
\draw[ultra thick] (-1.4,-1.4) -- (1.4,1.4);
\draw[ultra thick] (-1.4,1.4) -- (1.4,-1.4);

\draw[fill=black] (0,0) circle (.0cm) node[anchor= south,xshift=1pt] {\small $x'$};
\draw[fill=black] \NE circle (.0cm) node[anchor= west] {$-2'$};
\draw[fill=black] \SE circle (.0cm) node[anchor= west] {$-1'$};
\draw[fill=black] \SW circle (.0cm) node[anchor= east] {$2'$};
\draw[fill=black] \NW circle (.0cm) node[anchor= east] {$1'$};

\node at (0,0) [circle,draw=black,fill=white,inner sep=0pt,minimum size=.15cm]  {};
\node at \NE [circle,fill=black,inner sep=0pt,minimum size=.15cm]  {};
\node at \NW [circle,fill=black,inner sep=0pt,minimum size=.15cm]  {};
\node at \SW [circle,fill=black,inner sep=0pt,minimum size=.15cm]  {};
\node at \SE [circle,fill=black,inner sep=0pt,minimum size=.15cm]  {};

\end{scope}

\end{tikzpicture}

\end{minipage}
\hfill
\begin{minipage}[c]{.6\textwidth}
Since $\mathcal{D}_{\bigotimes(\LL)}$ is a packing then, by Lemma \ref{lem:relations} \ref{lem:relations;tangencydisk}, the boundaries of $d_t$ and $d_{t'}$ do not intersect. Therefore, $d_t$ and $d_{t'}$ are disjoint as well as $b_t$ and $b_{t'}$. Hence, $\mathcal{R}\cap \mathcal{R}'=\emptyset$.
\end{minipage}\\
\vspace{1cm}\\
$n=1$\hspace{0.7cm}
\begin{minipage}[c]{.25\textwidth}
    \begin{tikzpicture}[scale=.45]
\draw (-1,-1) -- (-3,-1) -- (-3,1) -- (3,1) -- (3,-1) -- (1,-1) -- (0,1) -- (-1,-1);
\draw[ultra thick] (-3.4,-1.4) -- (-2,0) -- (0,1) --(2,0)-- (3.4,-1.4);
\draw[ultra thick] (-3.4,1.4) -- (-0.6,-1.4);
\draw[ultra thick] (3.4,1.4) -- (0.6,-1.4);

\draw[fill=black] (-2,0) circle (.0cm) node[anchor= south] {$x$};
\draw[fill=black] (0,1) circle (.0cm) node[anchor= south] {$1=1'$};
\draw[fill=black] (-1,-1) circle (.0cm) node[anchor= west] {$2$};
\draw[fill=black] (-3,1) circle (.0cm) node[anchor= east] {$-1$};
\draw[fill=black] (-3,-1) circle (.0cm) node[anchor= east] {$-2$};

\draw[fill=black] (2,0) circle (.0cm) node[anchor= south] {$x'$};
\draw[fill=black] (1,-1) circle (.0cm) node[anchor= east] {$2'$};
\draw[fill=black] (3,1) circle (.0cm) node[anchor= west] {$-2'$};
\draw[fill=black] (3,-1) circle (.0cm) node[anchor= west] {$-1'$};

\node at (-2,0) [circle,draw=black,fill=white,inner sep=0pt,minimum size=.15cm]  {};
\node at (-3,-1) [circle,fill=black,inner sep=0pt,minimum size=.15cm]  {};
\node at (-3,1) [circle,fill=black,inner sep=0pt,minimum size=.15cm]  {};
\node at (0,1) [circle,fill=black,inner sep=0pt,minimum size=.15cm]  {};
\node at (-1,-1) [circle,fill=black,inner sep=0pt,minimum size=.15cm]  {};
\node at (2,0) [circle,draw=black,fill=white,inner sep=0pt,minimum size=.15cm]  {};
\node at (3,-1) [circle,fill=black,inner sep=0pt,minimum size=.15cm]  {};
\node at (3,1) [circle,fill=black,inner sep=0pt,minimum size=.15cm]  {};
\node at (1,-1) [circle,fill=black,inner sep=0pt,minimum size=.15cm]  {};

\end{tikzpicture}      
\end{minipage}
\hfill
\begin{minipage}[c]{.6\textwidth}
The (possibly empty) region $d_t\cap d_{t'}$ must be contained in $d_1$ so $b_t\cap b_{t'}$ is contained in $b_1$. As a consequence, $\intt(\mathcal{R})\cap \intt(\mathcal{R}')=\emptyset$.
\end{minipage}\\
\vspace{1cm}\\
$n=2$\hspace{0.7cm}
\begin{minipage}[c]{.25\textwidth}
\centering
\begin{tikzpicture}[scale=.45]
\draw (-2,-1) rectangle (2,1);
\draw (0,1) -- (0,-1);
\draw[ultra thick] (-2.4,-1.4) -- (0,1)--(2.4,-1.4);
\draw[ultra thick] (-2.4,1.4) -- (0,-1) -- (2.4,1.4);

\draw[fill=black] (-1,0) circle (.0cm) node[anchor= south] {$x$};
\draw[fill=black] (0,1) circle (.0cm) node[anchor= south] {$1=1'$};
\draw[fill=black] (0,-1) circle (.0cm) node[anchor= north] {$2=2'$};
\draw[fill=black] (-2,1) circle (.0cm) node[anchor= east] {$-1$};
\draw[fill=black] (-2,-1) circle (.0cm) node[anchor= east] {$-2$};

\draw[fill=black] (1,0) circle (.0cm) node[anchor= south] {$x'$};
\draw[fill=black] (2,1) circle (.0cm) node[anchor= west] {$-2'$};
\draw[fill=black] (2,-1) circle (.0cm) node[anchor= west] {$-1'$};

\node at (-1,0) [circle,draw=black,fill=white,inner sep=0pt,minimum size=.15cm]  {};
\node at (-2,-1) [circle,fill=black,inner sep=0pt,minimum size=.15cm]  {};
\node at (-2,1) [circle,fill=black,inner sep=0pt,minimum size=.15cm]  {};
\node at (0,1) [circle,fill=black,inner sep=0pt,minimum size=.15cm]  {};
\node at (0,-1) [circle,fill=black,inner sep=0pt,minimum size=.15cm]  {};
\node at (1,0) [circle,draw=black,fill=white,inner sep=0pt,minimum size=.15cm]  {};
\node at (2,-1) [circle,fill=black,inner sep=0pt,minimum size=.15cm]  {};
\node at (2,1) [circle,fill=black,inner sep=0pt,minimum size=.15cm]  {};

\end{tikzpicture}        
    
\end{minipage}
\hfill
\begin{minipage}[c]{.6\textwidth}
We can apply a standard transformation to get a standard disk packing $[\mathcal{D}_\boxtimes\cup\mathcal{D}_\boxtimes']^1_2$ where the disks $d_1$, $d_2$, $d_t$ and $d_{t'}$ become half-spaces as in Figure \ref{fig:proofn=2}. 
\end{minipage}

\begin{figure}[H]
\centering

\begin{tikzpicture}[scale=.5]
\newcommand\XXmax{4.4}

\newcommand\YY{1.5}
\newcommand\YYmin{-2.5}

\begin{scope}[xshift=-5cm]
\fill[darkgray,opacity=\opac] (0,3.64) circle (1.81cm);
\fill[darkgray,opacity=\opac] (0,-0.26) circle (2.09cm);
\fill[darkgray,opacity=\opac](-4.59,.23) circle (2.53cm);
\fill[darkgray,opacity=\opac] (-2.81,3.3) circle (1.02cm);
\fill[darkgray,opacity=\opac] (-1.8,1.88) circle (.71cm);
\fill[darkgray,opacity=\opac] (-1.74,1.68) circle (1.75cm);
\fill[darkgray,opacity=\opac](4.81,.6) circle (2.8cm);
\fill[darkgray,opacity=\opac] (2.3,2.9) circle (0.6cm);
\fill[darkgray,opacity=\opac] (1.66,1.87) circle (.61cm);
\fill[darkgray,opacity=\opac] (1.52,1.56) circle (1.54cm);

\node at (-3.32,2.42) [circle,fill=black,inner sep=0pt,minimum size=.1cm]  {};
\node at (-1.8,3.42) [circle,fill=black,inner sep=0pt,minimum size=.1cm]  {};
\node at (0,1.83) [circle,fill=black,inner sep=0pt,minimum size=.1cm]  {};
\node at (-2.08,-0.04) [circle,fill=black,inner sep=0pt,minimum size=.1cm]  {};
\node at (-.02,1.4) [circle,fill=black,inner sep=0pt,minimum size=.1cm]  {};
\node at (2.06,0.11) [circle,fill=black,inner sep=0pt,minimum size=.1cm]  {};
\node at (2.75,2.49) [circle,fill=black,inner sep=0pt,minimum size=.1cm]  {};
\node at (1.72,3.09) [circle,fill=black,inner sep=0pt,minimum size=.1cm]  {};

\draw (0,3.64) circle (1.81cm) node[anchor= south, inner sep=2pt]  {$d_1$};
\draw (0,-0.26) circle (2.09cm) node[anchor=north,inner sep=2pt] {$d_{2}$};
\draw (-4.59,.23) circle (2.53cm) node {$d_{-1}$};
\draw (-2.81,3.3) circle (1.02cm) node {$d_{-2}$};
\draw (-1.8,1.88) circle (.71cm) node {$d_{x}$};
\draw[densely dotted] (-1.74,1.68) circle (1.75cm) node at (-.9,.9) { $d_{t}$};
\draw (4.81,.6) circle (2.8cm) node {$d_{-1'}$};
\draw (2.3,2.9) circle (0.6cm) node at (2.4,2.9) {\scriptsize $d_{-2'}$};
\draw (1.66,1.87) circle (.61cm) node {$d_{x'}$};
\draw[densely dotted] (1.52,1.56) circle (1.54cm) node at (1,.9) {$d_{t'}$};
\end{scope}

\begin{scope}[xshift=12cm,yshift=1cm,scale=2.3]
\fill [darkgray,opacity=\opac]  (-2.8,1) rectangle (\XXmax,1.5);
\fill [darkgray,opacity=\opac]  (-2.8,-1) rectangle (\XXmax,-1.5);
\fill [darkgray,opacity=\opac]  (-2.8,1.5) -- (-.63,1.5) -- (0.13,-1.5) -- (-2.8,-1.5); 
\fill [darkgray,opacity=\opac]  (2.12,1.5) -- (\XXmax,1.5) -- (\XXmax,-1.5) -- (1.59,-1.5);

\fill[darkgray,opacity=\opac](0,-.29) circle (.71cm);
\fill[darkgray,opacity=\opac] (-1.69,0) circle (1cm);
\fill[darkgray,opacity=\opac] (-.5,0.65) circle (.35cm);
\fill[darkgray,opacity=\opac] (1.68,-0.36) circle (.64cm);
\fill[darkgray,opacity=\opac] (3.28,0) circle (1cm);
\fill[darkgray,opacity=\opac] (2.03,.61) circle (.39cm);

\node at (-.5,1) [circle,fill=black,inner sep=0pt,minimum size=.1cm] {};
\node at (-.34,.34) [circle,fill=black,inner sep=0pt,minimum size=.1cm] {};
\node at (0,-1) [circle,fill=black,inner sep=0pt,minimum size=.1cm] {};
\node at (1.68,-1) [circle,fill=black,inner sep=0pt,minimum size=.1cm]  {};
\node at (2.03,1) [circle,fill=black,inner sep=0pt,minimum size=.1cm]  {};
\node at (1.9,0.24) [circle,fill=black,inner sep=0pt,minimum size=.1cm]  {};

\draw (-2.8,1) -- (\XXmax,1) node at (-2.5,1.2) {$d_{1}$};
\draw (-2.8,-1) -- (\XXmax,-1) node at (-2.5,-1.2) {$d_{2}$};
\draw (0,-.29) circle (.71cm) node[anchor= west]  {$d_{-1}$};
\draw (-.5,0.65) circle (.35cm) node {\small $d_{-2}$};
\draw (-1.69,0) circle (1cm) node {$d_{x}$};

\draw (1.68,-0.36) circle (.64cm) node[anchor= east,inner sep=-2pt] {$d_{-1'}$};
\draw (2.03,.61) circle (.39cm) node {\small $d_{-2'}$};
\draw (3.28,0) circle (1cm) node {$d_{x'}$};
\draw[densely dotted] (-.63,1.5) -- (0.13,-1.5) node at (-.8,1.3) {$d_{t}$};
\draw[densely dotted] (2.12,1.5) -- (1.59,-1.5) node at (2.4,1.3) {$d_{t'}$};

\end{scope}

\node at (-5,-4) {(a)};
\node at (14,-4) {(b)};
\end{tikzpicture}  
   \caption{(a) $\mathcal{D}_\boxtimes\cup\mathcal{D}_\boxtimes'$ in the case $n=2$ together with their tangency disks;  (b)$[\mathcal{D}_\boxtimes\cup\mathcal{D}_\boxtimes']^1_2$.}
\label{fig:proofn=2}
\end{figure} 

\noindent
The lines $\partial d_t$ and $\partial d_{t'}$ in $[\mathcal{D}_\boxtimes\cup\mathcal{D}_\boxtimes']^1_2$ either intersect in a point lying in $d_1\cup d_2$ or they are parallel implying, in both cases, that the region $d_t\cap d_{t'}$ is contained in $d_1\cup d_2$. Therefore, $b_t\cap b_{t'}$ is contained in $b_1\cup b_2$ and thus $\intt(\mathcal{R})\cap \intt(\mathcal{R}')=\emptyset$.

\newpage
\noindent
$n=3$
\begin{minipage}[c]{.25\textwidth}
\centering
  
\begin{tikzpicture}[scale=.45]
\draw (0,2) --(0,-2);
\draw (0,2) -- (-3,0) -- (0,-2) -- (3,0) -- (0,2);
\draw[ultra thick] (-3.5,0) -- (3.5,0);
\draw[ultra thick] (0,2) -- (1.6,0)--(0,-2) -- (-1.6,0) -- (0,2);

\draw[fill=black] (-1.6,0) circle (.0cm) node[anchor= south] {\scriptsize $x$};
\draw[fill=black] (0,2) circle (.0cm) node[anchor= south] {$1=1'$};
\draw[fill=black] (0,0) circle (.0cm) node[anchor= south] {\scriptsize $2=2'$};
\draw[fill=black] (0,-2) circle (.0cm) node[anchor= north] {$-1=-1'$};
\draw[fill=black] (-3.5,0) circle (.0cm) node[anchor= south] {$-2$};
\draw[fill=black] (1.6,0) circle (.0cm) node[anchor= south,xshift=1pt] {\scriptsize $x'$};
\draw[fill=black] (3.5,0) circle (.0cm) node[anchor= south] {$-2'$};

\node at (0,0) [circle,fill=black,inner sep=0pt,minimum size=.15cm]  {};
\node at (-1.6,0) [circle,draw=black,fill=white,inner sep=0pt,minimum size=.15cm]  {};
\node at (-3,0) [circle,fill=black,inner sep=0pt,minimum size=.15cm]  {};
\node at (0,2) [circle,fill=black,inner sep=0pt,minimum size=.15cm]  {};
\node at (0,-2) [circle,fill=black,inner sep=0pt,minimum size=.15cm]  {};
\node at (1.6,0) [circle,draw=black,fill=white,inner sep=0pt,minimum size=.15cm]  {};
\node at (3,0) [circle,fill=black,inner sep=0pt,minimum size=.15cm]  {};

\end{tikzpicture}       
    
\end{minipage}
\hfill
\begin{minipage}[c]{.7\textwidth}
The boundaries of $d_t$ and $d_{t'}$ intersect at the tangency points of $d_1\cap d_2$ and $d_{-1}\cap d_2$, see Figure \ref{fig:proofn=3}. Therefore $d_t\cap d_t'$ is contained in $d_2$ which implies that   $b_t\cap b_t'$ is contained in $b_2$ and hence $\intt(\mathcal{R})\cap \intt(\mathcal{R}')=\emptyset$.
\end{minipage}

\begin{figure}[H]
\centering

\begin{tikzpicture}[scale=.5]
\newcommand\XXmax{4.4}

\newcommand\YY{1.5}
\newcommand\YYmin{-2.5}

\begin{scope}[xshift=-5cm,scale=0.2]
\fill[darkgray,opacity=\opac] (0,12.06) circle (10.29cm);
\fill[darkgray,opacity=\opac] (0,0) circle (1.76cm);
\fill[darkgray,opacity=\opac](0,-10.25) circle (8.49cm);
\fill[darkgray,opacity=\opac] (-24.54,-1.29) circle (17.63cm);
\fill[darkgray,opacity=\opac] (-4.35,-.07) circle (2.59cm);
\fill[darkgray,opacity=\opac] (-7.17,0) circle (7.38cm);
\fill[darkgray,opacity=\opac] (24.54,-1.29) circle (17.63cm);
\fill[darkgray,opacity=\opac] (4.35,-.07) circle (2.59cm);
\fill[darkgray,opacity=\opac] (7.17,0) circle (7.38cm);

\node at (-9.04,7.14) [circle,fill=black,inner sep=0pt,minimum size=.1cm]  {};
\node at (-7.97,-7.34) [circle,fill=black,inner sep=0pt,minimum size=.1cm]  {};
\node at (0,-1.76) [circle,fill=black,inner sep=0pt,minimum size=.1cm]  {};
\node at (0,1.76) [circle,fill=black,inner sep=0pt,minimum size=.1cm]  {};
\node at (7.97,-7.34) [circle,fill=black,inner sep=0pt,minimum size=.1cm]  {};
\node at (9.04,7.14) [circle,fill=black,inner sep=0pt,minimum size=.1cm]  {};

\draw (0,12.06) circle (10.29cm) node {$d_1$};
\draw (0,0) circle (1.76cm) node {\scriptsize $d_{2}$};
\draw (0,-10.25) circle (8.49cm) node {$d_{-1}$};
\draw (-24.54,-1.29) circle (17.63cm) node {$d_{-2}$};
\draw (-4.35,-.07) circle (2.59cm) node {\scriptsize $d_{x}$};
\draw[densely dotted] (-7.17,0) circle (7.38cm) node[anchor=east,inner sep=.3cm] { $d_{t}$};
\draw (24.54,-1.29) circle (17.63cm) node {$d_{-2'}$};
\draw (4.35,-.07) circle (2.59cm) node {\scriptsize $d_{x'}$};
\draw[densely dotted] (7.17,0) circle (7.38cm) node[anchor=west,inner sep=.3cm] { $d_{t'}$};
\end{scope}

\begin{scope}[xshift=12cm,scale=2.3]
\fill [darkgray,opacity=\opac]  (-2.8,1) rectangle (2.8,1.5);
\fill [darkgray,opacity=\opac]  (-2.8,-1) rectangle (2.8,-1.5);
\fill [darkgray,opacity=\opac]  (-2.8,1.5) -- (-.63,1.5) -- (0.13,-1.5) -- (-2.8,-1.5); 
\fill [darkgray,opacity=\opac]  (2.8,1.5) -- (.63,1.5) -- (-0.13,-1.5) -- (2.8,-1.5);

\fill[darkgray,opacity=\opac](0,-.29) circle (.71cm);
\fill[darkgray,opacity=\opac] (-1.69,0) circle (1cm);
\fill[darkgray,opacity=\opac] (-.5,0.65) circle (.35cm);
\fill[darkgray,opacity=\opac] (1.69,0) circle (1cm);
\fill[darkgray,opacity=\opac] (.5,0.65) circle (.35cm);

\node at (-.5,1) [circle,fill=black,inner sep=0pt,minimum size=.1cm] {};
\node at (-.34,.34) [circle,fill=black,inner sep=0pt,minimum size=.1cm] {};
\node at (0,-1) [circle,fill=black,inner sep=0pt,minimum size=.1cm] {};

\node at (.5,1) [circle,fill=black,inner sep=0pt,minimum size=.1cm]  {};
\node at (.34,.34) [circle,fill=black,inner sep=0pt,minimum size=.1cm]  {};

\draw (-2.8,1) -- (2.8,1) node at (-2.5,1.2) {$d_{1}$};
\draw (-2.8,-1) -- (2.8,-1) node at (-2.5,-1.2) {$d_{2}$};
\draw (0,-.29) circle (.71cm) node[anchor= south]  {$d_{-1}$};
\draw (-.5,0.65) circle (.35cm) node {\small $d_{-2}$};
\draw (-1.69,0) circle (1cm) node {$d_{x}$};

\draw (.5,0.65) circle (.35cm) node {\small $d_{-2'}$};
\draw (1.69,0) circle (1cm) node {$d_{x'}$};
\draw[densely dotted] (-.63,1.5) -- (0.13,-1.5) node at (-.8,1.3) {$d_{t}$};
\draw[densely dotted] (.63,1.5) -- (-.13,-1.5) node at (.8,1.3) {$d_{t'}$};

\end{scope}

\node at (-5,-5) {(a)};
\node at (12,-5) {(b)};
\end{tikzpicture}  
   \caption{(a) $\mathcal{D}_\boxtimes\cup\mathcal{D}_\boxtimes'$ in the case $n=3$ together with their tangency disks;  (b)$[\mathcal{D}_\boxtimes\cup\mathcal{D}_\boxtimes']^1_2$.}
\label{fig:proofn=3}
\end{figure} 

\vspace{1cm}

\noindent
$n=4$
\begin{minipage}[c]{.25\textwidth}
\centering
\begin{tikzpicture}[scale=.45]
\draw (0,3) --(0,-3);
\draw (0,3) ..controls (4,2) and (4,-2) .. (0,-3);
\draw[ultra thick] (0,.7) -- (1.6,0)--(0,-.7) -- (-1.6,0) -- (0,.7);
\draw[ultra thick] (0,3) -- (1.6,0)--(0,-3) -- (-1.6,0) -- (0,3);

\draw[fill=black] (-1.6,0) circle (.0cm) node[anchor= east] {$x$};
\draw[fill=black] (0,.7) circle (.0cm) node[anchor= south,xshift=1pt,yshift=-1pt] {\scriptsize $1=1'$};
\draw[fill=black] (0,-.7) circle (.0cm) node[anchor= north,xshift=1pt,yshift=2pt] {\scriptsize $2=2'$};
\draw[fill=black] (0,-3) circle (.0cm) node[anchor= north,xshift=1pt,yshift=1pt] {$-1=-1'$};
\draw[fill=black] (0,3) circle (.0cm) node[anchor= south] {$-2=-2'$};
\draw[fill=black] (1.6,0) circle (.0cm) node[anchor= west] {$x'$};

\node at (0,.7) [circle,fill=black,inner sep=0pt,minimum size=.15cm]  {};
\node at (0,-.7) [circle,fill=black,inner sep=0pt,minimum size=.15cm]  {};
\node at (-1.6,0) [circle,draw=black,fill=white,inner sep=0pt,minimum size=.15cm]  {};
\node at (0,3) [circle,fill=black,inner sep=0pt,minimum size=.15cm]  {};
\node at (0,-3) [circle,fill=black,inner sep=0pt,minimum size=.15cm]  {};
\node at (1.6,0) [circle,draw=black,fill=white,inner sep=0pt,minimum size=.15cm]  {};
\end{tikzpicture}           
    
\end{minipage}
\begin{minipage}[r]{.7\textwidth}
In this case, the tangency graph of $\mathcal{D}_\boxtimes\cup\mathcal{D}_\boxtimes'$ is isomorphic to the octahedral graph by taking $x'=3$ and $x=-3$, see Figure \ref{fig:proofn=4}. We have that $d_t=-d_{t'}$ which implies that $b_t$ and $b_{t'}$ are externally tangent. Thus,  $\intt(\mathcal{R})\cap \intt(\mathcal{R}')=\emptyset$.

\begin{flushright}$\square$\end{flushright}
\end{minipage}

\begin{figure}[H]
\centering

\begin{tikzpicture}[scale=.5]
\newcommand\XXmax{4.4}

\newcommand\YY{1.5}
\newcommand\YYmin{-2.5}

\begin{scope}[xshift=-2cm,scale=2.7]
\clip (-1.6,-1.6) rectangle  (1.6,1.6);

\fill[darkgray,opacity=\opac] (-1,1.732) circle (1.732cm);
\fill[darkgray,opacity=\opac] (-1,-1.732) circle (1.732cm);
\fill[darkgray,opacity=\opac] (0.101,-0.175) circle (.175cm);
\fill[darkgray,opacity=\opac] (0.101, 0.175) circle (.175cm);
\fill[darkgray,opacity=\opac] (-0.202,0) circle (.175cm) ;
\fill[darkgray,opacity=\opac] (-4,-4) rectangle (4,4);
\fill[darkgray,opacity=\opac] (2,0) circle (1.732cm);

\node at (-1,0) [circle,fill=black,inner sep=0pt,minimum size=.1cm]  {};
\node at (.101,0) [circle,fill=black,inner sep=0pt,minimum size=.1cm]  {};
\node at (0., 0.317837) [circle,fill=black,inner sep=0pt,minimum size=.1cm]  {};
\node at (0.,-0.317837) [circle,fill=black,inner sep=0pt,minimum size=.1cm]  {};

\draw (-1,1.732) circle (1.732cm) node[anchor=north,inner sep=0.5cm] {$d_1$};
\draw (-1,-1.732) circle (1.732cm) node[anchor=south,inner sep=0.5cm] { $d_{2}$};
\draw (0.101,-0.175) circle (.175cm) node {\tiny $d_{-1}$};
\draw (0.101, 0.175) circle (.175cm) node {\tiny $d_{-2}$};
\draw (-0.202,0) circle (.175cm) node {\small $d_{x}$};
\draw[densely dotted] (-.45,0) circle (.55cm) node at (-.3,.4) {\small $d_{t}$};
\node at (-.05,.6) {\small $d_{t'}$};
\draw (2,0) circle (1.732cm) node[anchor=east,inner sep=0.7cm] { $d_{x'}$};

\end{scope}

\begin{scope}[xshift=12cm,scale=2.3]
\fill [darkgray,opacity=\opac]  (-2.6,1) rectangle (2.6,1.5);
\fill [darkgray,opacity=\opac]  (-2.6,-1.5) rectangle (2.6,-1);
\fill [darkgray,opacity=\opac]  (-2.6,-1.5) rectangle (2.6,1.5);

\fill[darkgray,opacity=\opac](0,-.5) circle (.5cm);
\fill[darkgray,opacity=\opac] (-1.414,0) circle (1cm);
\fill[darkgray,opacity=\opac] (0,.5) circle (.5cm);
\fill[darkgray,opacity=\opac] (1.414,0) circle (1cm);

\node at (0,1) [circle,fill=black,inner sep=0pt,minimum size=.1cm] {};
\node at (0,0) [circle,fill=black,inner sep=0pt,minimum size=.1cm] {};
\node at (0,-1) [circle,fill=black,inner sep=0pt,minimum size=.1cm] {};

\draw (-2.6,1) -- (2.6,1) node at (-2.3,1.2) {$d_{1}$};
\draw (-2.6,-1) -- (2.6,-1) node at (-2.3,-1.2) {$d_{2}$};
\draw (0,-.5) circle (.5cm) node[anchor=east,inner sep=0pt]  {\scriptsize $d_{-1}$};
\draw (0,0.5) circle (.5cm) node[anchor=east,inner sep=0pt] {\scriptsize $d_{-2}$};
\draw (-1.414,0) circle (1cm) node {$d_{x}$};

\draw (1.414,0) circle (1cm) node {$d_{x'}$};
\draw[densely dotted] (0,1.5) -- (0,-1.5) node at (-.2,1.3) {$d_{t}$};
\node at (.25,1.3) {$d_{t'}$};

\end{scope}

\node at (-2,-5) {(a)};
\node at (12,-5) {(b)};
\end{tikzpicture}  
   \caption{(a) $\mathcal{D}_\boxtimes\cup\mathcal{D}_\boxtimes'$ in the case $n=4$ together with their tangency disks;  (b)$[\mathcal{D}_\boxtimes\cup\mathcal{D}_\boxtimes']^1_2$.}
   
\label{fig:proofn=4}
\end{figure} 

\vspace{1cm}

We notice that our method requires two balls for each bridge in order to connect the closest pair. Unfortunately, this cannot be done with a single ball since it would be too large to be contained in the crossing region, a central request in our proof.\\

\newpage
\section{The necklace algorithm}\label{sec:alg}
In this section we present an algorithm arisen from the constructive proof of our main result. The necklaces representations of the Figures \ref{fig:neck731} and \ref{fig:neck817} are based on this algorithm.\\

The balls are given in inversive coordinates. Instead of computing the bridge balls by using the basis with the mirror ball $b_c^*$ (as done in the proof of Lemma \ref{region}), we use the basis $\mathcal{B}=\{b_x,b_c,b_f,b_{-c},b_{\epsilon z}\}$ which avoids the computation of $b_c^*$. To this end, we need the inversive products $\lambda_{-c,3}:=\langle b_{-c}, b_{\epsilon3}\rangle=\langle b_{c},b_{\epsilon3}'\rangle$ and  $\lambda_{z,3}:=\langle b_{\epsilon z}, b_{\epsilon3}\rangle=\langle b_{\epsilon z}, b_{\epsilon3}'\rangle$. These values are given in the proof of Lemma \ref{region} in terms of the smaller standard curvature by $\lambda_{-c,3}=-1-2\sqrt{\kappa}$ and $\lambda_{z,3}=\sqrt{3+2\sqrt{\kappa}-\kappa}$. The smaller standard curvature can be computed by using Lemma \ref{lem:relations} \ref{lem:relations;ltok} obtaining $\kappa=\frac{1-\lambda_{c,-c}}{2}=\frac{8}{1-\lambda_{f,-f}}$ where $\lambda_{c,-c}:=\langle b_c,b_{-c}\rangle$ and  $\lambda_{f,-f}:=\langle b_f,b_{-f}\rangle$.
In order to obtain a disk packing from the tangency graph we use the well-known algorithm of Collins and Stephenson given in \cite{collinsstephenson} where the radius of the outer disks and the visual precision can be chosen. In all our examples we set the outer radii to be equal to $1$ and precision $10^{-4}$. \\

\begin{table}[h]
\caption{The necklace algorithm.}
\hrule
\vspace{0.1cm}
\begin{itemize}
      
\item[] \textbf{Input:} A link diagram $\LL$ with $n$ crossings of a link $L$.
\item[] \textbf{Output:} A necklace representation $\mathcal{B}_\LL$ of the link $L$ with $5n$ balls.
\item[] \textbf{Algorithm:}
\begin{enumerate}[label*=\arabic*.,itemsep=5pt]
\item Construct the pyramidal patchwork $\bigotimes(\mathcal{L})=(V_\times\cup V_\bigcirc,E_\otimes)$
\item Construct a disk packing $\mathcal{D}_{\bigotimes(\mathcal{L})}$ of tangency graph $\bigotimes(\LL)$
\item Construct a ball packing $\mathcal{B}_{\bigotimes(\mathcal{L})}$ obtained by blowing-up $\mathcal{D}_{\bigotimes(\mathcal{L})}$
\item Set $\mathcal{B}_{\bigwedge(\LL)}=\{\}$, $Q=\mathsf{diag}(1,1,1,1,-1)$, 
$b_z=\begin{pmatrix}
0&0&1&0&0
\end{pmatrix}^T$
\item \textbf{For} $x\in V_\times$ \textbf{do:}
\begin{enumerate}[label=(\alph*),itemsep=3pt]
\item Give to $\mathcal{B}_\boxtimes(x)$ a pyramid labeling $\mathcal{B}_\boxtimes(x)=\{b_x,b_1,b_2,b_{-1},b_{-2}\}$
\item Compute the inversive product $\lambda= b_1^T Q b_{-1}$
\item \textbf{If} $\lambda\ge -3$ \textbf{then:}
\begin{enumerate}
\item[i.] $B=(b_x|b_1|b_2|b_{-1}|b_{z})$, $\kappa= \frac{1-\lambda}{2}$
\end{enumerate}
\item \textbf{else:}
\begin{enumerate}
\item[i.] $B=(b_x|b_2|b_1|b_{-2}|b_{z})$, $\kappa= \frac{8}{1-\lambda}$
\end{enumerate}
\item $\lambda_{-c,3}=-1-2\sqrt{\kappa}$,
 $\lambda_{z,3}=\sqrt{3+2\sqrt{\kappa}-\kappa}$
\item $b_3(x)= 
(\begin{pmatrix}
-1&-1&-1&\lambda_{-c,3}&\lambda_{z,3}
\end{pmatrix}B^{-1}Q)^T$
\item $b_3'(x)=
(\begin{pmatrix}
-1&\lambda_{-c,3}-1&-1&\lambda_{z,3}
\end{pmatrix}B^{-1}Q)^T$
\item \textbf{If} the thread made by the bridge balls is under-crossing in $\mathcal{L}$ \textbf{then:}
\begin{enumerate}
\item[i.] $b_3(x)\leftarrow\mathsf{diag}\begin{pmatrix}
1&1&-1&1&1
\end{pmatrix}b_3(x)$
\item[ii.] $b_3'(x)\leftarrow\mathsf{diag}\begin{pmatrix}
1&1&-1&1&1
\end{pmatrix}  b_3'(x)
$ 
\end{enumerate}
\item $\mathcal{B}_{\bigwedge(\LL)}\leftarrow\mathcal{B}_{\bigwedge(\LL)}\cup \{b_3(x),b_3'(x)\}$
\end{enumerate}
\item $\mathcal{B}_\LL=\mathcal{B}_{\bigotimes(\LL)}\cup \mathcal{B}_{\bigwedge(\LL)}$
\end{enumerate}
\end{itemize}
\hrule
\end{table}

We believe that this algorithm can be useful to investigate 3D invariants of links such as the \textit{rope length} or the \textit{stick number}. We end with the following.
\begin{cor} Let $K$ be a non-trivial knot. Then, there is a non-planar $3$-ball packable graph on $5\mathsf{cr}(K)$ vertices admitting a piece-wise linear embedding in $\mathbb{R}^3$ and containing a Hamiltonian cycle  ambient isotopic to $K$.
\end{cor}
\begin{proof} 
By applying the Necklace algorithm to a minimal crossing diagram of $K$ we obtain a ball packing $\mathcal{B}_K$. Let $H$ be its tangency graph. By construction, $|V(H)|=5\mathsf{cr}(K)$. Since $K$ is non-trivial $H$ has at least one crossing vertex $x$. Let $H_x$ be the subgraph of $H$ induced by $x$ and its $6$ neighbors. It is easy to see that $H_x$ contains the complete graph of $5$ vertices as a minor and thus, by Kuratowski theorem, $H$ is non planar. Furthermore, the carrier of $\mathcal{B}_K$ is a piece-wise linear embedding of $H$ in $\mathbb{R}^3$. The cycle corresponding to the thread of the necklace representation of $K$ in $\mathcal{B}_K$ is a Hamiltonian cycle of $H$.

\end{proof}
\newpage

\begin{minipage}[l]{.35\textwidth}
\hspace{1.2cm}
{\centering
\includegraphics[width=0.45\textwidth]{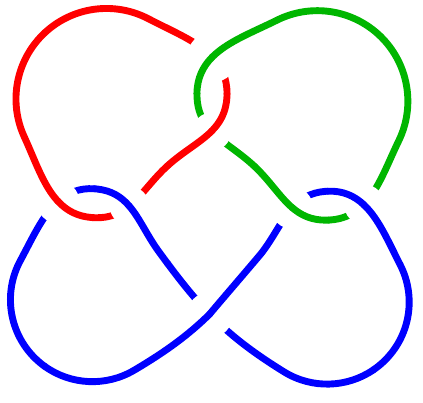}
}
\vspace{-.3cm}

{\tiny
\begin{tabular}[t]{|c|ccc|c|}
\multicolumn{5}{c}{Link $7_{1}^3$ } \\
\hline
Ball& $x$& $y$& $z$& $r$\\
\hline

  1 & 0. & 0. & 0. & 1. \\
 2 & 0.4068 & 1. & -0.1882 & 0.0958 \\
 3 & 0.519 & 1.083 & -0.1184 & 0.0603 \\
 4 & 0.6344 & 1.0947 & 0. & 0.1054 \\
 5 & 0.7338 & 1.0234 & 0.0655 & 0.0334 \\
 6 & 0.7762 & 0.9686 & 0.071 & 0.0362 \\
 7 & 0.8407 & 0.7983 & 0. & 0.1593 \\
 8 & 1. & 0.5458 & 0. & 0.1392 \\
 9 & 2. & 0. & 0. & 1. \\
 10 & 1.4814 & 0.9131 & 0.1095 & 0.0558 \\
 11 & 1.4327 & 0.9925 & 0.0815 & 0.0415 \\
 12 & 1.3656 & 1.0947 & 0. & 0.1054 \\
 13 & 1.3204 & 0.9813 & -0.0604 & 0.0307 \\
 14 & 1.2858 & 0.9279 & -0.065 & 0.0331 \\
 15 & 1.1593 & 0.7983 & 0. & 0.1593 \\
 16 & 1.1253 & 0.6243 & 0.1733 & 0.0886 \\
 17 & 1. & 0.4638 & 0.2589 & 0.1323 \\
 18 & 2. & 2. & 0. & 1. \\
 19 & 1.0949 & 1.4501 & 0.1314 & 0.0671 \\
 20 & 1. & 1.3938 & 0.0956 & 0.0489 \\
 21 & 0.8873 & 1.3285 & 0. & 0.1127 \\
 22 & 1. & 1.211 & 0. & 0.0501 \\
 23 & 1.1302 & 1.0863 & 0. & 0.1302 \\
 24 & 1.2695 & 0.9745 & 0. & 0.0485 \\
 25 & 1.3864 & 0.9006 & 0. & 0.0898 \\
 26 & 1.546 & 1. & 0. & 0.0982 \\
 27 & 0. & 2. & 0. & 1. \\
 28 & 1. & 1.5198 & 0. & 0.1093 \\
 29 & 1.1127 & 1.3285 & 0. & 0.1127 \\
 30 & 1.0501 & 1.2136 & 0.0693 & 0.0354 \\
 31 & 1. & 1.162 & 0.0714 & 0.0365 \\
 32 & 0.8698 & 1.0863 & 0. & 0.1302 \\
 33 & 0.7305 & 0.9745 & 0. & 0.0485 \\
 34 & 0.6136 & 0.9006 & 0. & 0.0898 \\
 35 & 0.454 & 1. & 0. & 0.0982 \\
\hline
\end{tabular}
}\\

\end{minipage}
\begin{minipage}[r]{.6\textwidth}
\begin{figure}[H]
\captionsetup{width=1\linewidth}
    \includegraphics[width=1\textwidth]{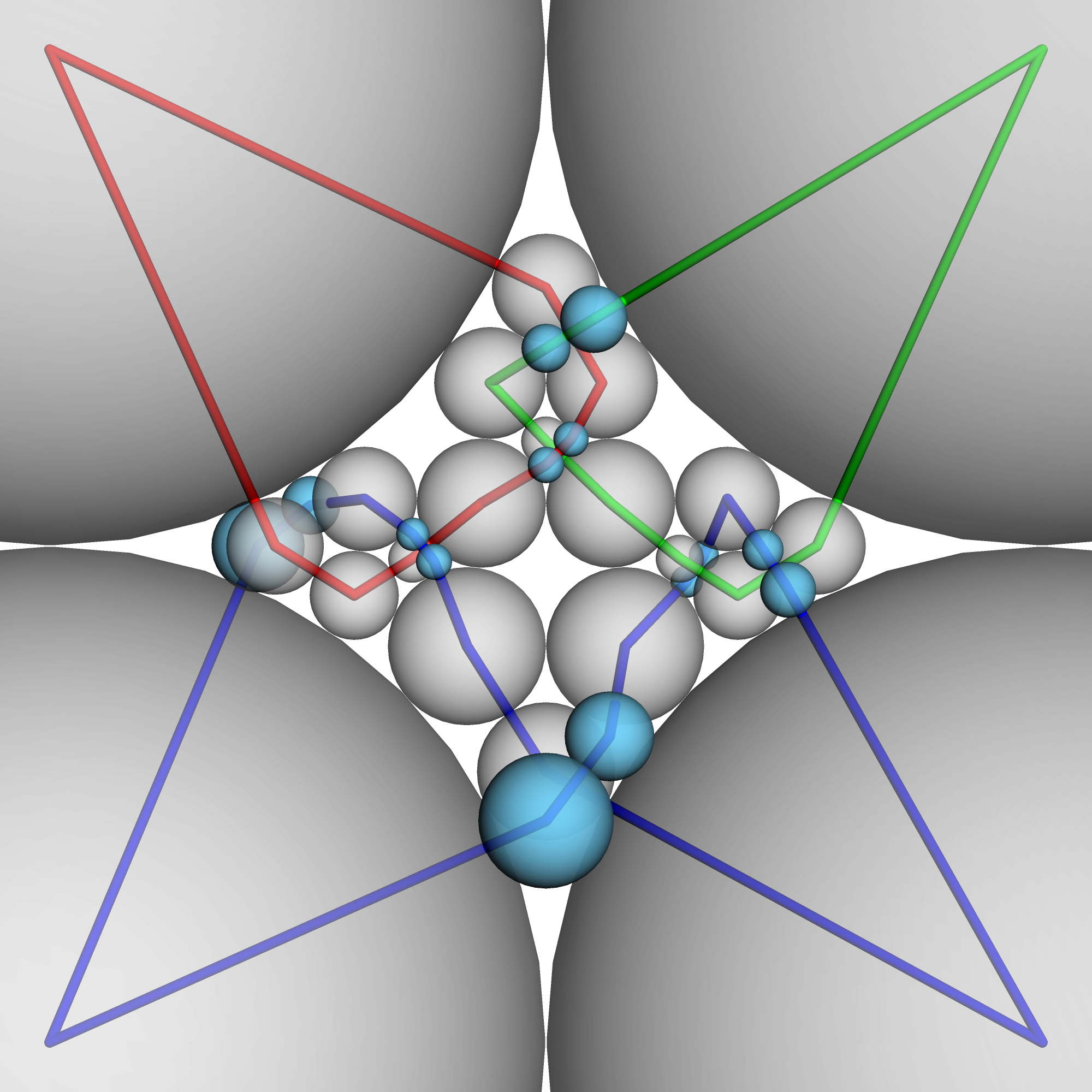}
    \caption{Necklace representation of the link $7^3_1$ with $35$ balls.}
    \label{fig:neck731}
\end{figure}
\end{minipage}
\vfill

\begin{minipage}[l]{.35\textwidth}
\hspace{1cm}
{\centering
\includegraphics[width=0.45\textwidth]{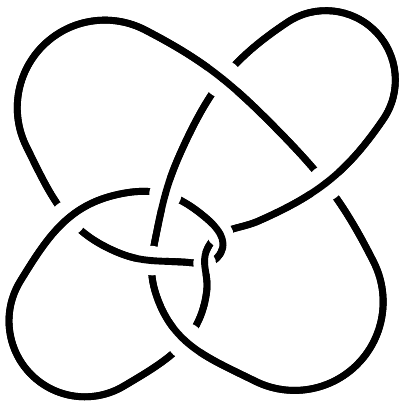}
}
\vspace{-.2cm}

{\tiny
\begin{tabular}[t]{|c|ccc|c|}
\multicolumn{5}{c}{Knot $8_{17}$ } \\
\hline
Ball& $x$& $y$& $z$& $r$\\
\hline
1 & 0. & 0. & 0. & 1. \\
 2 & 0.5168 & 0.958 & 0.2146 & 0.1094 \\
 3 & 0.6549 & 1.0442 & 0.1381 & 0.0704 \\
 4 & 0.7919 & 1.0508 & 0. & 0.1243 \\
 5 & 0.9356 & 1.0528 & -0.0671 & 0.0343 \\
 6 & 0.9982 & 1.0317 & -0.0625 & 0.0319 \\
 7 & 1.081 & 0.9676 & 0. & 0.09 \\
 8 & 1.1502 & 0.8904 & 0.0455 & 0.0232 \\
 9 & 1.164 & 0.8497 & 0.0395 & 0.0202 \\
 10 & 1.1595 & 0.7925 & 0. & 0.0494 \\
 11 & 1.0879 & 0.7759 & 0. & 0.0241 \\
 12 & 1.0007 & 0.7799 & 0. & 0.0633 \\
 13 & 0.902 & 0.7895 & 0. & 0.0358 \\
 14 & 0.7559 & 0.8166 & 0. & 0.1127 \\
 15 & 0.5757 & 0.9527 & 0. & 0.1132 \\
 16 & 0.1772 & 1.9921 & 0. & 1. \\
 17 & 1.1772 & 1.4541 & 0. & 0.1356 \\
 18 & 1.3308 & 1.212 & 0. & 0.1511 \\
 19 & 1.4928 & 1.1678 & -0.1645 & 0.084 \\
 20 & 1.6347 & 1.0364 & -0.2477 & 0.1265 \\
 21 & 2. & 0. & 0. & 1. \\
 22 & 1. & 0.4186 & 0.2005 & 0.1025 \\
 23 & 0.9104 & 0.5354 & 0.1241 & 0.0634 \\
 24 & 0.895 & 0.6486 & 0. & 0.1053 \\
 25 & 0.876 & 0.7684 & -0.0535 & 0.0273 \\
 26 & 0.8848 & 0.8205 & -0.0502 & 0.0256 \\
 27 & 0.9271 & 0.9005 & 0. & 0.0779 \\
 28 & 0.9571 & 1.0187 & 0. & 0.044 \\
 29 & 1.0243 & 1.2064 & 0. & 0.1554 \\
 30 & 1.0559 & 1.3765 & -0.1686 & 0.0861 \\
 31 & 1.1772 & 1.533 & -0.2525 & 0.1289 \\
 32 & 2.1772 & 1.9921 & 0. & 1. \\
 33 & 1.5588 & 1.0432 & 0. & 0.1327 \\
 34 & 1.3111 & 0.9157 & 0. & 0.1459 \\
 35 & 1.1441 & 0.8682 & 0. & 0.0278 \\
 36 & 1.0843 & 0.8388 & 0. & 0.0389 \\
 37 & 1.0718 & 0.7954 & 0.0305 & 0.0156 \\
 38 & 1.0717 & 0.7615 & 0.0372 & 0.019 \\
 39 & 1.1051 & 0.6481 & 0. & 0.1049 \\
 40 & 1. & 0.4677 & 0. & 0.1039 \\
\hline
\end{tabular}
}\\

\end{minipage}
\begin{minipage}[r]{.6\textwidth}
\begin{figure}[H]
\captionsetup{width=1\linewidth}
    \includegraphics[width=1\textwidth]{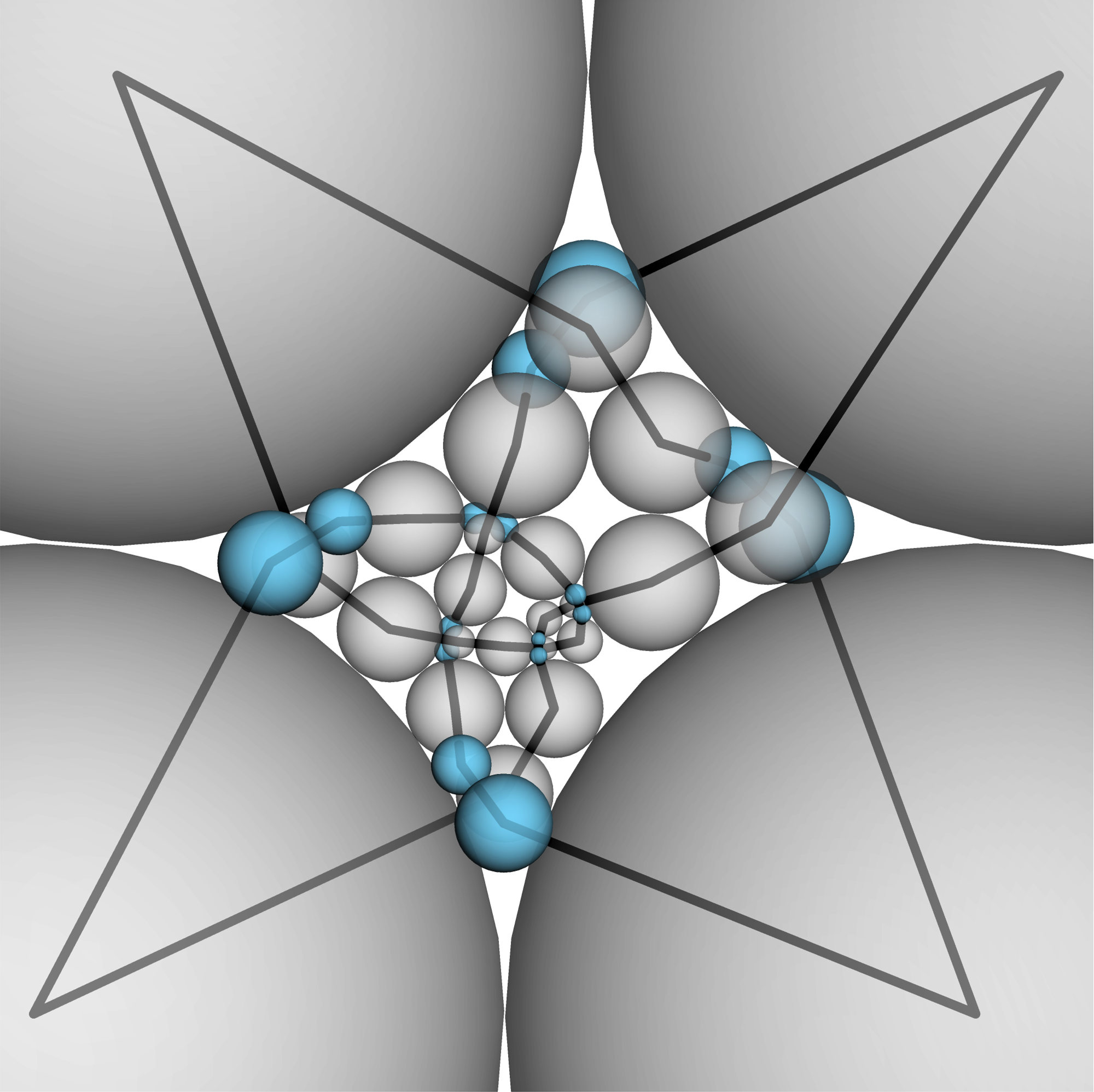}
    \caption{Necklace representation of the knot $8_{17}$ with 40 balls.}
    \label{fig:neck817}
\end{figure}
\end{minipage}

{\bf Acknowledgements}.  We would like to thank the referees for helpful remarks which improved the readability of the paper.

\bibliographystyle{amsplain}

\end{document}